\documentclass[11pt, reqno]{amsart}
\usepackage{amssymb, amstext, amscd, amsmath, amsthm}
\usepackage{color}
\usepackage{hyperref}

\usepackage{xy}
\xyoption{all}

\theoremstyle{plain}
\newtheorem{theorem}{Theorem}[section]
\newtheorem{thm}[theorem]{Theorem}
\newtheorem{cor}[theorem]{Corollary}
\newtheorem{prop}[theorem]{Proposition}
\newtheorem{lem}[theorem]{Lemma}
\newtheorem*{theorem*}{Theorem}
\theoremstyle{definition}
\newtheorem{rem}[theorem]{Remark}
\newtheorem{defn}[theorem]{Definition}

\newcommand{\bN}{{\mathbb{N}}}

\newcommand{\bR}{{\mathbb{R}}}

\newcommand{\bT}{{\mathbb{T}}}

\newcommand{\bZ}{{\mathbb{Z}}}

  \newcommand{\B}{{\mathcal{B}}}
  \newcommand{\C}{{\mathcal{C}}}
  \newcommand{\D}{{\mathcal{D}}}
  
  \newcommand{\F}{{\mathcal{F}}}
  \newcommand{\G}{{\mathcal{G}}}

  \newcommand{\I}{{\mathcal{I}}}

\newcommand{\M}{{\mathcal{M}}}
  
\renewcommand{\O}{{\mathcal{O}}}

\renewcommand{\S}{{\mathcal{S}}}
  \newcommand{\T}{{\mathcal{T}}}

\newcommand{\fA}{{\mathfrak{A}}}

\newcommand{\fk}{{\mathfrak{k}}}
\newcommand{\fs}{{\mathfrak{s}}}
\newcommand{\ft}{{\mathfrak{t}}}

\newcommand{\upchi}{{\raise.35ex\hbox{\ensuremath{\chi}}}}

\newcommand{\qforal}{\quad\text{for all}\quad}

\newcommand{\Aut}{\operatorname{Aut}}

\newcommand{\Iso}{\operatorname{Iso}}

\newcommand{\ca}{\mathrm{C}^*}

\newcommand{\mt}{\varnothing}

\newcommand{\Per}{{\rm Per}}

\begin{document}
\title[KMS States of Self-similar $k$-Graph C*-algebras]{KMS States of Self-similar $k$-Graph C*-algebras}
\author[H. Li]{Hui Li}
\address{Hui Li,
Research Center for Operator Algebras and Shanghai Key Laboratory of Pure Mathematics and Mathematical Practice, Department of Mathematics, East China Normal University, 3663 Zhongshan North Road, Putuo District, Shanghai 200062, China}
\email{lihui8605@hotmail.com}
\author[D. Yang]{Dilian Yang}
\address{Dilian Yang,
Department of Mathematics $\&$ Statistics, University of Windsor, Windsor, ON
N9B 3P4, CANADA}
\email{dyang@uwindsor.ca}

\thanks{The first author was supported by Research Center for Operator Algebras of East China Normal University and by Science and Technology Commission of Shanghai Municipality (STCSM), grant No. 18dz2271000.}
\thanks{The second author was partially supported by an NSERC Discovery Grant.}

\begin{abstract}
Let $G$ be a countable discrete amenable group, and $\Lambda$ be a strongly connected finite $k$-graph.
If $(G,\Lambda)$ is a pseudo free and locally faithful self-similar action which satisfies the finite-state condition, then the structure of the KMS simplex of the
C*-algebra $\O_{G,\Lambda}$  associated to $(G,\Lambda)$
is described: it is either empty or affinely isomorphic to the tracial state space of the C*-algebra of the periodicity group $\Per_{G,\Lambda}$ of $(G,\Lambda)$,
depending on whether the Perron-Frobenius eigenvector of $\Lambda$ preserves the $G$-action. 
As applications of our main results, we also exhibit several classes of important  examples.
\end{abstract}


\maketitle

\section{Introduction}

Self-similar $k$-graph C*-algebras were initially studied by the authors in \cite{LY17_2}. These algebras embrace many known important C*-algebras as special classes, such as $k$-graph C*-algebras introduced by Kumjian-Pask \cite{KP00}, Exel-Pardo algebras \cite{EP17}, unital Katsura algebras \cite{Kat08_1}, and Nekrashevych algebras \cite{Nek09}. Roughly speaking, a self-similar $k$-graph C*-algebra $\O_{G,\Lambda}$ is a universal C*-algebra generated by a unitary representation $u$ of $G$ and a Cuntz-Krieger representation $s$ of a $k$-graph $\Lambda$, which are compatible with the ambient self-similar action of $G$ on $\Lambda$:
$u_g s_\mu=s_{g \cdot \mu} u_{g \vert_\mu}$ for all $g\in G$ and $\mu\in\Lambda$.

The study of KMS states has been attracting a lot of attention. See, for example, \cite{ABLS18, HLRS14, HLRS15, LRRW14, LRRW16} to name just a few, and the references therein.
For a self-similar $k$-graph C*-algebra $\O_{G,\Lambda}$, there is a natural gauge action $\gamma: \prod_{i=1}^k\bT\to \Aut(\O_{G,\Lambda})$ such that
$\gamma_z(s_\mu)=z^{d(\mu)}s_\mu$ and $\gamma_z(u_g)=u_g$ for all $z\in \prod_{i=1}^k\bT$, $\mu\in\Lambda$, $g\in G$.
Let $r\in \prod_{i=1}^k\bR$. Then there is an action $\alpha$ of $\bR$ on $\O_{G,\Lambda}$ such that $\alpha_t=\gamma_{e^{itr}}$.
So it is natural to study the KMS states of the one-parameter dynamical system $(\O_{G,\Lambda}, \bR, \alpha)$. It turns
out that, for $0<\beta<\infty$, the existence of a KMS$_\beta$ state uniquely determines $r\in \prod_{i=1}^k\bR$.
Therefore, in this paper we assume that $\beta=1$ (see Section \ref{SS:basicKMS}), and compute the KMS$_1$ states of the preferred dynamical system,
which are simply called the \textit{KMS states of $\O_{G,\Lambda}$}. Based on the results in \cite{HLRS14, HLRS15, LRRW14, LRRW16}, one can not expect to compute the KMS states of $\O_{G,\Lambda}$ in general.
Here, we suppose that $\Lambda$ is a strongly connected finite $k$-graph, and that the self-similar action $(G,\Lambda)$ is pseudo free, locally faithful, and satisfies the finite-state condition (see Definition \ref{D:pf}).
Under these natural conditions, it turns out that the KMS simplex of $\O_{G,\Lambda}$ has a rather nice and neat structure,
which unifies some known results in \cite{HLRS14, HLRS15, LRRW14, LRRW16}.

Usually, one starts with computing the KMS states of the corresponding Toeplitz-type algebras,
and then checks when they can be factorized through the Cuntz-Krieger-type algebras. In this paper, our strategy is different:
we attack the Cuntz-Krieger-type algebras $\O_{G,\Lambda}$ directly.
To achieve our goal, we take advantage of the following features of a self-similar action $(G,\Lambda)$:
(1) the $G$-periodicity $\Per_{G,\Lambda}$ of $(G,\Lambda)$ forms a subgroup of $\bZ^k$ (Theorem \ref{T:Pergroup});
(2) there is a canonical Cartan subalgebra $\M_{G,\Lambda}$, which we call the
self-similar cycline subalgebra of $\O_{G,\Lambda}$ (Theorem \ref{T:Exp});
and (3) $\M_{G,\Lambda}$ is isomorphic to the tensor product of the group C*-algebra $\ca(\Per_{G,\Lambda})$ and the diagonal subalgebra $\D_\Lambda$
of the $k$-graph C*-algebra $\O_\Lambda$ (Theorem \ref{T:tensor}).
Then our structure theorem, Theorem \ref{T:KMSPer}, states that the KMS simplex of $\O_{G,\Lambda}$ is affinely isomorphic to the tracial state space
of the periodicity group C*-algebra $\ca(\Per_{G,\Lambda})$ if the Perron-Frobenius eigenvector of $\Lambda$ preserves the $G$-action
on the vertex set $\Lambda^0$ (see Remark \ref{R:KMSass}) .
Two interesting applications of our main results are also given (Corollaries \ref{C:Per} and \ref{C:uniKMS}).

In order to obtain feature (1), we thoroughly analyze the $G$-periodicity theory of the self-similar action $(G,\Lambda)$ in Sections \ref{S:chaGape} and
\ref{S:perthe}. Our results also unify some periodicity theory for $k$-graphs \cite{HLRS15, RS07, Yan15}.
To get feature (2), we borrow our approach from \cite{BLY17}, and invoke the nice criterion of Cartan subalgebras for groupoids in \cite{BNRSW16}.
Feature (3) heavily depends on the detailed descriptions of the KMS simplex of $\O_{G,\Lambda}$ given in Theorems \ref{T:kms} and \ref{T:Mstate}.

Several classes of examples are exhibited in Section \ref{S:eg} as concrete applications of our main results.

We should mention that our approaches here are much motivated by \cite{HLRS15} although different.

\smallskip

Let us end the introduction by fixing our notation and conventions.

\subsection*{Notation and conventions}

Denote by $\mathbb{N}$ the set of all nonnegative integers. 
Let $k$ be a fixed positive integer, which is allowed to be $\infty$. 
Denote by $\{e_i\}_{i=1}^{k}$ the standard basis of $\mathbb{N}^k$. For $n,m \in \mathbb{N}^k$, denote by $n \lor m$ (resp.~$n\wedge m$) the coordinatewise maximum (resp.~minimum) of $n$ and $m$.
We also use the multi-index notation: $z^n:=\prod_{i=1}^{k}z_i^{n_i}$ for all $z \in \prod_{i=1}^k\bT$ and $n\in \bN^k$.
The identity of a group $G$ is denoted as $1_G$.

\section{Preliminaries}

Some necessary background is provided in this section. The main sources are \cite{HLRS15, KP00, LY17_2}.

\subsection{$k$-graph C*-algebras}

A countable small category $\Lambda$ is called a \emph{$k$-graph} if there exists a functor $d:\Lambda \to \mathbb{N}^k$ satisfying that for $\mu\in\Lambda, n,m \in \mathbb{N}^k$ with $d(\mu)=n+m$, there exist unique $\beta\in d^{-1}(n)$ and $\alpha\in d^{-1}(m) $ such that $\mu=\beta\alpha$.
 A functor $f:\Lambda_1 \to \Lambda_2$ is called a \emph{graph morphism} if $d_2 \circ f=d_1$.

Let $\Lambda^n:=d^{-1}(n)$.
Then $\Lambda$ is said to be \emph{row-finite} if $\vert v\Lambda^{n}\vert<\infty$ for all $v \in \Lambda^0$ and $n \in \mathbb{N}^k$;
\textit{finite} if $|\Lambda^n|<\infty$ for all $n\in \bN^k$; \emph{source-free} if $v\Lambda^{n} \neq \mt$ for all $v \in \Lambda^0$ and $n \in \mathbb{N}^k$;
 and \emph{strongly connected} if $v \Lambda w \neq \mt$ for all $v,w \in \Lambda^0$.

A very important row-finite source-free $k$-graph is $\Omega_k$ which is defined as follows.  Let $\Omega_k:=\{(p,q) \in \mathbb{N}^k \times \mathbb{N}^k:p \leq q\}$. For $(p,q), (q,m) \in \Omega_k$, define $(p,q) \cdot (q,m):=(p,m)$, $r(p,q):=(p,p)$, $s(p,q):=(q,q)$, and $d(p,q):=q-p$.
An \textit{infinite path} $x$ of $\Lambda$ is a graph morphism from $\Omega_k$ to $\Lambda$.
The set of all infinite paths of $\Lambda$ is denoted by $\Lambda^\infty$.

For a row-finite source-free $k$-graph $\Lambda$, the \textit{$k$-graph C*-algebra $\O_\Lambda$} is the universal C*-algebra generated by a Cuntz-Krieger $\Lambda$-family, 
which is a family of partial isometries $\{s_\lambda:\lambda\in\Lambda\}$ satisfying
\begin{enumerate}
\item[(CK1)] $\{s_v\}_{v \in \Lambda^0}$ is a family of mutually orthogonal projections;
\item[(CK2)] $s_{\mu\nu}=s_{\mu} s_{\nu}$ if $s(\mu)=r(\nu)$;
\item[(CK3)] $s_{\mu}^* s_{\mu}=s_{s(\mu)}$ for all $\mu \in \Lambda$; and
\item[(CK4)] $s_v=\sum_{\mu \in v \Lambda^{n}}s_\mu s_\mu^*$ for all $v \in \Lambda^0, n \in \mathbb{N}^k$.
\end{enumerate}
Let $\D_\Lambda:=\ca(\{s_\mu s_\mu^*:\mu\in\Lambda\})$ be the \emph{diagonal subalgebra} of $\O_\Lambda$.

From now on,  all $k$-graphs are assumed to be row-finite and source-free.

\subsection{The Perron-Frobenius theory for $k$-graphs}

In this subsection, we recall the Perron-Frobenius theory of strongly connected finite $k$-graphs from \cite{HLRS15}. 
It is worth mentioning that, as shown in \cite[Lemma 2.1]{HLRS15}, a strongly connected $k$-graph is always source-free and sink-free.

Let $\Lambda$ be a strongly connected finite $k$-graph. 
For $T \in M_{\Lambda^0}(\mathbb{N})$, let $\rho(T)$ denote the spectral radius of $T$. For $p\in\bN^k$, let $T_p\in M_{\Lambda^0}(\mathbb{N})$
be defined by $T_p(v,w):=\vert v\Lambda^p w\vert$ for $v,w \in \Lambda^0$.
In particular, $T_{e_i}$ ($1\le i\le k$) are called the \textit{coordinate matrices of $\Lambda$}.
Let $\rho(\Lambda):=(\rho(T_{e_i}))_{i=1}^{k}$ be a vector in $\prod_{i=1}^k\bR$.

\begin{thm}[{\cite{HLRS15}}]
\label{T:PF}
Let $\Lambda$ be a strongly connected finite $k$-graph. Then the following properties hold true. 
\begin{enumerate}
\item $\rho(T_{p})=\rho(\Lambda)^p>0$ for all $p \in \mathbb{N}^k$.
\item There exists a unique $x_\Lambda \in \ell^1(\Lambda^0)$, called the \emph{Perron-Frobenius eigenvector} of $\Lambda$, satisfying that
\begin{itemize}
\item $x_\Lambda(v)>0$ for all $v \in \Lambda^0$;
\item $x_\Lambda$ has the unit norm;
\item $T_{e_i}x_\Lambda=\rho(T_{e_i})x_\Lambda$ for all $1 \leq i \leq k$;
\item if $y \in \ell^1(\Lambda^0)$ satisfies $T_{e_i} y=\rho(T_{e_i})y$ for all $1 \leq i \leq k$, then $y \in \mathbb{C} x_\Lambda$.
\end{itemize}
\item For any $\lambda \in \prod_{i=1}^k[0,\infty)$ and $0 \neq y \in \ell^1(\Lambda^0)$,
 $y \geq 0$ and $T_{e_i} y \leq \lambda_i y$ for all $1 \leq i \leq k$ $\implies$ $y>0$ and $\lambda_i \geq \rho(T_{e_i})$; and
 $\lambda_i=\rho(T_{e_i})$ for all $1 \leq i \leq k$ $\iff$ $T_{e_i} y = \lambda_i y$ for all $1 \leq i \leq k$;
\item There exists a unique state $\phi_\Lambda$ on $\D_\Lambda$, called the \emph{Perron-Frobenius state} of $\Lambda$, such that
$\phi_\Lambda(s_\mu s_\mu^*)=\rho(\Lambda)^{-d(\mu)}x_\Lambda(s(\mu))$ for all $\mu\in\Lambda$.
\end{enumerate}
\end{thm}

\subsection{Self-similar $k$-graphs}

In this subsection, we recall the definitions of self-similar $k$-graphs and their associated C*-algebras introduced by the authors in \cite{LY17_2}, which are generalizations of Exel-Pardo C*-algebras for directed graphs (i.e., $1$-graphs) studied in \cite{EP17}.

Let $\Lambda$ be a $k$-graph. A bijection $\pi:\Lambda \to \Lambda$ is called an \emph{automorphism} of $\Lambda$ if
$\pi$ preserves the degree, source and range maps.
Denote by $\Aut(\Lambda)$ the automorphism group of $\Lambda$.

Let $G$ be a (discrete countable) group. We say that \textit{$G$ acts on $\Lambda$} if there is a group homomorphism $\varphi$ from $G$ to $\Aut(\Lambda)$.
For $g\in G$ and $\mu\in\Lambda$, we often simply write $\varphi_g(\mu)$ as $g\cdot \mu$.

\begin{defn}[{\cite[Definition~3.2]{LY17_2}}]
\label{D:ss}
Let $\Lambda$ be a $k$-graph, and $G$ be a group acting on $\Lambda$.
Then we call $(G,\Lambda)$ \emph{a self-similar action} if there exists a \emph{restriction map} $G\times \Lambda \to G$,
$(g,\mu)\mapsto g|_\mu$, such that
\begin{enumerate}
\item
$g\cdot (\mu\nu)=(g \cdot \mu)(g \vert_\mu \cdot \nu)$ for all $g \in G,\mu,\nu \in \Lambda$ with $s(\mu)=r(\nu)$.

\item
$g \vert_v =g$ for all $g \in G,v \in \Lambda^0$;

\item
$g \vert_{\mu\nu}=g \vert_\mu \vert_\nu$ for all $g \in G,\mu,\nu \in \Lambda$ with $s(\mu)=r(\nu)$;

\item
$1_G \vert_{\mu}=1_G$ for all $\mu \in \Lambda$;

\item
$(gh)\vert_\mu=g \vert_{h \cdot \mu} h \vert_\mu$ for all $g,h \in G,\mu \in \Lambda$.
\end{enumerate}
In this case, $\Lambda$ is also called a \textit{self-similar $k$-graph over $G$}.
\end{defn}

\begin{defn}
\label{D:pf}
A self-similar action $(G,\Lambda)$ is said to
\begin{enumerate}
\item be \emph{pseudo free} if there are $g \in G$ and $\mu \in \Lambda$ such that $g \cdot \mu=\mu$ and $g \vert_\mu=1_G $, then $g=1_G$;
\item be \emph{locally faithful} if
there exist $g\in G$ and $v\in \Lambda^0$ such that $g\cdot \mu=\mu$ for all $\mu\in v\Lambda$, then $g=1_G$;
\item satisfy the \emph{finite-state condition} if for any $g \in G$, the set $\{g \vert_\mu:\mu \in \Lambda\}$ is finite. 
\end{enumerate}
\end{defn}

\begin{rem} Some remarks are in order.
\begin{itemize}
\item[(1)] The local faithfulness condition is stronger than the faithfulness condition.
A simple and useful observation is the following: if there are $g\in G$ and $v\in\Lambda^0$ such that $g\cdot x = x$ for all $x\in v\Lambda^\infty$, then $g=1_G$ 
(see Subsection~\ref{SS:pathgrp}).

\item[(2)] If $(G,\Lambda)$ is a locally faithful self-similar action, then one can easily check that, in Definition \ref{D:ss}, Condition (i) implies Conditions (ii)-(v), and so only Condition (i) is required in this case.  This reconciles with the classical definition of self-similar groups in \cite{Nek05}.

\item[(3)] The pseudo freeness plays an important role in \cite{EP17, LY17_2} in the study of self-similar C*-algebras.

\item[(4)] For self-similar groups, the finite-state condition is closely related to the contracting property (\cite{Nek05}).
\end{itemize}
\end{rem}

\begin{defn}[{\cite[Definition~3.8]{LY17_2}}]\label{D:O}
Let $(G,\Lambda)$ be a self-similar action with $\vert\Lambda^0\vert<\infty$. The \textit{self-similar $k$-graph C*-algebra} $\mathcal{O}_{G,\Lambda}$
is defined to be the universal unital C*-algebra generated by a family of unitaries $\{u_g\}_{g \in G}$ and
a Cuntz-Krieger family $\{s_\mu\}_{\mu \in \Lambda}$ satisfying
\begin{itemize}
\item[(i)]
$u_{gh}=u_g u_h$ for all $g$ and $h \in G$;
\item[(ii)]
$u_g s_\mu=s_{g \cdot \mu} u_{g \vert_\mu}$ for all $g \in G$ and $\mu \in \Lambda$.
\end{itemize}
\end{defn}

Thus $\O_{G,\Lambda}$ is generated by a universal pair $(u,s)$ of representations, where $u$ is a unitary representation of $G$ and
$s$ is a representation of the $k$-graph C*-algebra $\O_\Lambda$, such that $u$ and $s$ are compatible with the underlying self-similar action.

Let us record the following result  (\cite[Propositions~3.12 and 5.10]{LY17_2}) for later use.

\begin{prop}
\label{P:genO}
Let $(G,\Lambda)$ be a self-similar action with $\vert\Lambda^0\vert<\infty$.
Then
\begin{enumerate}
\item
 the linear span of $\{s_\mu u_g s_\nu^*: \mu, \nu \in \Lambda, g \in G, s(\mu)= g \cdot s(\nu)\}$ is a dense $*$-subalgebra of $\mathcal{O}_{G,\Lambda}$;
\item
$\O_\Lambda$ naturally embeds into $\O_{G,\Lambda}$.
\end{enumerate}
\end{prop}

\subsection{Path groupoid $\G_{G,\Lambda}$ associated to $(G,\Lambda)$}
\label{SS:pathgrp}

To a self-similar $k$-graph $\Lambda$ over a group $G$, one can associate a path-like groupoid $\G_{G,\Lambda}$.  For this, we must introduce more notation.
Let $C(\mathbb{N}^k,G)$ be the set of all mappings from $\mathbb{N}^k$ to $G$, which is a group under the pointwise multiplication. For $z \in \mathbb{Z}^k$ and $f \in C(\mathbb{N}^k,G)$, define $\T_z(f) \in C(\mathbb{N}^k,G)$ by
\begin{align*}
\T_z(f)(p)&:= \begin{cases}
    f(p-z) &\text{if $p-z \in \mathbb{N}^k$} \\
    1_G &\text{if $p-z \notin \mathbb{N}^k$}.
\end{cases}
\end{align*}
Define an equivalence relation $\sim$ on $C(\bN^k,G)$ as follows: for $f,g \in C(\mathbb{N}^k,G)$,
\[
f \sim g \iff  \text{there exists }z \in \mathbb{N}^k \text{ such that } f(p)=g(p)\qforal p \geq z.
\]
Let $Q(\bN^k,G)$ be the quotient group $C(\mathbb{N}^k,G)/\!\!\!\sim$.  For $f\in C(\mathbb{N}^k,G)$, we write $[f] \in Q(\mathbb{N}^k,G)$.
Then $\T_z$ yields an automorphism, still denoted by $\T_z$, on $Q(\bN^k,G)$.
Moreover, $\T:\mathbb{Z}^k \to \Aut( Q(\bN^k,G)),\ z\mapsto \T_z$, is a homomorphism.
So one can form the semidirect product $Q(\bN^k,G) \rtimes_\T \mathbb{Z}^k$.

For $g \in G, x \in \Lambda^\infty$, define $g\cdot x\in \Lambda^\infty$ and $g|_x\in C(\mathbb{N}^k,G)$ by
\begin{align*}
(g \cdot x)(p,q)&:=g \vert_{x(0,p)} \cdot x(p,q) \qforal  p \leq q \in \mathbb{N}^k,\\
g \vert_x(p)&:=g \vert_{x(0,p)}\qforal p \in \mathbb{N}^k.
\end{align*}
Notice that $(g \cdot x)(0,q)=g \cdot (x(0,q))$. So we can write it as $g\cdot x(0,q)$ without any obscurity.

The \emph{self-similar path groupoid} associated to the self-similar $k$-graph $\Lambda$ over $G$ is defined as
\[
\G_{G,\Lambda}:=\left\{\big(\mu (g \cdot x);\T_{d(\mu)}([g \vert_{x}]),d(\mu)-d(\nu);\nu x\big):
\begin{matrix}
g \in G, \mu,\nu \in \Lambda ,\\
\text{ with } s(\mu)=g \cdot s(\nu)
\end{matrix}
\right\},
\]
which is a subgroupoid of $\Lambda^\infty \times (Q(\bN^k,G)\rtimes_\T \mathbb{Z}^k ) \times \Lambda^\infty$.
For $g \in G, \mu,\nu \in \Lambda$ with $s(\mu)=g \cdot s(\nu)$, define
\[
Z(\mu,g,\nu):=\big\{(\mu (g \cdot x);\T_{d(\mu)}([g \vert_{x}]),d(\mu)-d(\nu);\nu x):x \in s(\nu)\Lambda^\infty\big\}.
\]
Endow $\mathcal{G}_{G,\Lambda}$ with the topology generated by the basic open sets
\[
\mathcal{B}_{G,\Lambda}:=\big\{Z(\mu,g,\nu):g \in G, \mu,\nu \in \Lambda,s(\mu)=g \cdot s(\nu)\big\}.
\]

\begin{thm}[{\cite[Theorems~5.7 and 5.9]{LY17_2}}]
Let $(G,\Lambda)$ be a pseudo free self-similar action. Then
\begin{enumerate}
\item
$\mathcal{G}_{G,\Lambda}$ is an ample groupoid, and $\mathcal{B}_{G,\Lambda}$ consists of compact open bisections;
\item
$\O_{G,\Lambda}\cong\ca(\G_{G,\Lambda})$ provided that $G$ is amenable.
\end{enumerate}
\end{thm}


\section{Characterizations of The $G$-Aperiodicity of Self-Similar $k$-Graphs}
\label{S:chaGape}

Let $\Lambda$ be a pseudo free self-similar $k$-graph over a group $G$. Our main goal in this section is to characterize when $\Lambda$ is $G$-aperiodic in terms of finite paths and cycline triples.
These characterizations generalize some results on $k$-graphs in the literature (\cite{RS07, Yan15}), and will be frequently used later.

\begin{defn}
Let $\Lambda$ be a self-similar $k$-graph over $G$. For $\mu,\nu \in \Lambda,g \in G$ with $s(\mu)=g \cdot s(\nu)$, the triple $(\mu,g,\nu)$ is called \emph{cycline} if $\mu(g \cdot x)=\nu x$ for all $x \in s(\nu)\Lambda^\infty$.
\end{defn}

Let $\C_{G,\Lambda}$ denote the set of all cycline triples.
Cycline triples are a generalization of cycline pairs for $k$-graphs in \cite{BNR14} (see also \cite{Yan16}).
Clearly, $(\mu,\nu)$ is a cycline pair if and only if $(\mu, 1_G, \nu)$ is a cycline triple.  We
use $\C_\Lambda$ to denote the set of all cycline pairs for $\Lambda$.

 Also,  for every $\mu\in\Lambda$, $(\mu, 1_G, \mu)$ is obviously cycline. Such cycline triples are said to be \textit{trivial}.

\begin{defn}[{\cite[Definition 6.4]{LY17_2}}]
\label{D:Gape}
Let $\Lambda$ be a self-similar $k$-graph over $G$.
 An infinite path $x \in \Lambda^\infty$ is said to be \emph{$G$-aperiodic} if, for $g \in G,p,q \in \mathbb{N}^k$ with $g \neq 1_G$ or $p \neq q$, we have $\sigma^{p}(x) \neq g \cdot \sigma^q(x)$; otherwise, $x$ is called \emph{$G$-periodic}.  $\Lambda$ is said to be \emph{$G$-aperiodic} if, for any $v \in \Lambda^0$, there exists a $G$-aperiodic path
 $x\in v\Lambda^\infty$; and \emph{$G$-periodic} otherwise.
\end{defn}

\begin{lem}
\label{L:equi}
Let $(\mu, g, \nu)$ be a cycline triple. Then
\begin{enumerate}
\item $(\nu, g^{-1}, \mu)$ is cycline.

\item
$
\sigma^{d(\mu)}(x)=\sigma^{d(\nu)}(g\cdot x)
$
for all $x\in s(\nu)\Lambda^\infty$.
\end{enumerate}
\end{lem}

\begin{proof}
(i) follows from the definition. For (ii), notice that from $\mu g\cdot x=\nu x$ for all $x\in s(\nu)\Lambda^\infty$ one has
\[
\mu g\cdot x(0, d(\nu)) g|_{x(0,d(\nu)}\sigma^{d(\nu)}(x)=\nu x(0,d(\mu))\sigma^{d(\mu)}(x).
\]
So one has $\mu g\cdot x(0,d(\nu))=\nu x(0,d(\mu))$ and $\sigma^{d(\mu)}(x)=g|_{x(0,d(\nu)}\sigma^{d(\nu)}(x)=\sigma^{d(\nu)}(g\cdot x)$.
\end{proof}

\begin{lem}\label{L:xper}
Let $\Lambda$ be a pseudo free self-similar $k$-graph over $G$.
Then $x\in\Lambda^\infty$ is $G$-periodic if and only if $\sigma^n(x)$ is $G$-periodic for any $n\in \bN^k$.
\end{lem}

\begin{proof}
Obviously it is sufficient to show the necessity. For this, let $x\in\Lambda^\infty$ be $G$-periodic and $n\in\bN^k$.
Thus there are $p,q \in \mathbb{N}^k$ and $g \in G$ with $p \neq q$ or $g \neq 1_G$ such that $\sigma^p(x)=g\cdot \sigma^q(x)$.
Then $\sigma^{p+n}(x)=\sigma^n(g\cdot \sigma^q(x))=g|_{\sigma^q(x)(0,n)}\sigma^{q+n}(x)$.
If $p\ne q$, clearly $\sigma^n(x)$ is $G$-periodic.
If $p=q$, then $g\ne 1_G$. To a contrary, assume that $g|_{\sigma^q(x)(0,n)}=1_G$.
It follows from $\sigma^q(x)=g\cdot \sigma^q(x)$ that $\sigma^q(x)(0,n)=g\cdot \sigma^q(x)(0,n)$.
But the action is pseudo free, we have $g=1_G$, a contraction. Therefore, $\sigma^n(x)$ is also $G$-periodic in this case.
\end{proof}

The following proposition is a generalization of \cite[Lemma~3.2]{RS07}, which uses finite paths to characterize the $G$-aperiodicity of $(G,\Lambda)$.

\begin{prop}\label{P:Gape}
Let $\Lambda$ be a pseudo free self-similar $k$-graph over $G$.
Then $\Lambda$ is $G$-aperiodic if and only if, for any $\mu \in \Lambda, g \in G, p,q \in \mathbb{N}^k$ with $g \neq 1_G$ or $p \neq q$, we have
\begin{enumerate}
\item
if $p \neq q$ then there exists $\nu \in s(\mu)\Lambda$ such that $d(\nu) \geq p \lor q$ and $\nu(p,d(\nu)+p-( p \lor q)) \neq g \vert_{(\mu\nu)(q,d(\mu)+q)} \cdot \nu(q,d(\nu)+q-(p \lor q))$;
\item
if $p=q$, then for any $\nu \in s(\mu)\Lambda$ satisfying that $d(\nu) \geq p$ and that $g \vert_{(\mu\nu)(p,d(\mu)+p)} \neq 1_G$, there exists $\gamma \in s(\nu)\Lambda$ such that $\nu(p,d(\nu))\gamma \neq g \vert_{(\mu\nu)(p,d(\mu)+p)} \cdot (\nu(p,d(\nu))\gamma)$.
\end{enumerate}
\end{prop}

\begin{proof}
The proof of ``Only if": Fix $\mu \in \Lambda, g \in G, p,q \in \mathbb{N}^k$ with $g \neq 1_G$ or $p \neq q$.

First suppose that $p\ne q$. Since $\varphi$ is $G$-aperiodic, there exists a $G$-aperiodic path $x$ at $s(\mu)$. By Lemma~\ref{L:xper}, $\mu x$ is $G$-aperiodic as well. Then $\sigma^{d(\mu)+p}(\mu x) \neq g \vert_{(\mu x)(q,d(\mu)+q)} \cdot \sigma^{d(\mu)+q}(\mu x)$. So $\sigma^{p}( x) \neq g \vert_{(\mu x)(q,d(\mu)+q)} \cdot \sigma^{q}(x)$. It follows immediately that there exists $l \in \mathbb{N}^k$ such that $\sigma^{p}( x)(0,l) \neq g \vert_{(\mu x)(q,d(\mu)+q)} \cdot \sigma^{q}(x)(0,l)$. Then $\nu:=x(0,l+p \lor q)$ satisfies all
properties required in (i).

Now suppose that $p=q$ (and so $g\ne 1_G$). Fix $\nu \in s(\mu)\Lambda$ satisfying that $d(\nu) \geq p$ and that $g \vert_{(\mu\nu)(p,d(\mu)+p)} \neq 1_G$. Since $\varphi$ is $G$-aperiodic, there exists a $G$-aperiodic path $y$ at $s(\nu)$. Let $x:=\nu y$. By Lemma~\ref{L:xper}, $\mu x$ is $G$-aperiodic as well. Then $\sigma^{d(\mu)+p}(\mu x) \neq g \vert_{(\mu x)(p,d(\mu)+p)} \cdot \sigma^{d(\mu)+p}(\mu x)$. So there exists $l \geq d(\nu)-p$ such that $\sigma^{d(\mu)+p}(\mu x)(0,l) \neq g \vert_{(\mu x)(p,d(\mu)+p)} \cdot \sigma^{d(\mu)+p}(\mu x)(0,l)$. Then $\gamma:=y(0,p+l-d(\nu))$ satisfies the property required in (ii).

\smallskip
The proof of ``If":
We enumerate the countable set $\{(g,$ $p,q)\in G \times \mathbb{N}^k \times \mathbb{N}^k: g \neq 1_G \text{ or } p \neq q\}$ as $\{(g^n,p^n,q^n)\}_{n=1}^{\infty}$. Fix $v \in \Lambda^0$. Let $\mu^0:=v$. By our assumptions there exists a sequence $(\mu^n)_{n=0}^{\infty} \subseteq \Lambda$ satisfying that for $n \geq 1$,
\begin{itemize}
\item[(1)] $s(\mu^{n-1})=r(\mu^{n})$;
\item[(2)] $d(\mu^n) \geq p^n \lor q^n$;
\item[(3)] if $p^n \neq q^n$ then
$\mu^n(p^n,d(\mu^n)+p^n-(p^n \lor q^n))
 \neq \newline g^n  \vert_{(\mu^1 \cdots \mu^n)(q^n,d(\mu^0\cdots\mu^{n-1})+q^n)} \cdot \mu^n(q^n, d(\mu^n)+q^n-(p^n \lor q^n));
$
\item[(4)] if $p^n =q^n$ and $g^n  \vert_{(\mu^1 \cdots \mu^n)(p^n,d(\mu^0\cdots\mu^{n-1})+p^n)} \neq 1_G$, then
\[
\mu^n(p^n,d(\mu^n)) \neq g^n  \vert_{(\mu^1 \cdots \mu^n)(p^n,d(\mu^0\cdots\mu^{n-1})+p^n)} \cdot \mu^n(p^n, d(\mu^n)).
\]
\end{itemize}
The sequence $(\mu^n)_{n=1}^{\infty}$ uniquely determines an infinite path $x \in \Lambda^\infty$. Fix $g \in G,p,q \in \mathbb{N}^k$ with $g \neq 1_G$ or $p \neq q$. Then there exists $N \geq 1$ such that $(g,p,q)=(g^N,p^N,q^N)$. We aim to show that $\sigma^{p^N}(x) \neq g^N \cdot \sigma^{q^N}(x)$ and we split into three cases.

\underline{Case 1:} $p^N \neq q^N$. We calculate that
\begin{align*}
&\sigma^{p^N}(x)(d(\mu^0\cdots\mu^{N-1}),d(\mu^1\cdots\mu^{N})-(p^N \lor q^N))\\
&=x(d(\mu^0\cdots\mu^{N-1})+p^N,d(\mu^1\cdots\mu^{N})+p^N-(p^N \lor q^N))\\
&=\mu^N(p^N,d(\mu^N)+p^N-(p^N \lor q^N))\\
&\neq g^N \vert_{(\mu^1 \cdots \mu^N)(q^N,d(\mu^0\cdots\mu^{N-1})+q^N)} \cdot\mu^N(q^N,d(\mu^N)+q^N-(p^N \lor q^N))\\
&\quad (\text{by Condition }(3))\\
&=(g^N \cdot \sigma^{q^N}(x))(d(\mu^0\cdots\mu^{N-1}),d(\mu^1\cdots\mu^{N})-(p^N \lor q^N)).
\end{align*}

\underline{Case 2:} $p^N=q^N$ and $g^N \vert_{(\mu^1 \cdots \mu^N)(p^N,d(\mu^0\cdots\mu^{N-1})+p^N)} \neq 1_G$. The proof is similar to Case 1 by using Condition (4).

\underline{Case 3:} $p^N=q^N$ and $g^N \vert_{(\mu^1 \cdots \mu^N)(p^N,d(\mu^0\cdots\mu^{N-1})+p^N)} =1_G$. Notice that $g^N \neq 1_G$. Suppose that $\sigma^{p^N}(x) =g^N \cdot \sigma^{p^N}(x)$ for a contradiction. Then $\sigma^{p^N}(x)(0$, $d(\mu^0\cdots\mu^{N-1})) =g^N \cdot \sigma^{p^N}(x)(0,d(\mu^0\cdots\mu^{N-1}))$. So
\[
g^N \cdot ((\mu^1 \cdots \mu^N)(p^N, d(\mu^0\cdots\mu^{N-1})+p^N))=(\mu^1 \cdots \mu^N)(p^N,d(\mu^0\cdots\mu^{N-1})+p^N).
\]
Since $\varphi$ is pseudo free, $g^N=1_G$ which is a contradiction. Therefore $\sigma^{p^N}(x) \neq g^N \cdot \sigma^{p^N}(x)$.

Therefore, $\sigma^{p^N}(x) \neq g^N \cdot \sigma^{q^N}(x)$ and $\Lambda$ is $G$-aperiodic.
\end{proof}

\begin{prop}\label{P:Gapecyc}
Let $\Lambda$ be a pseudo free self-similar action over $G$.
Then $\Lambda$ is $G$-aperiodic if and only if
every cycline triple is trivial.
\end{prop}

\begin{proof}
First of all, suppose that $\Lambda$ is $G$-aperiodic. Let $(\mu,g,\nu)$ be cycline. Since $\Lambda$ is $G$-aperiodic, there exists a $G$-aperiodic path $x\in s(\nu)\Lambda^\infty$. We calculate that
\begin{align*}
\sigma^{d(\mu)}(x)=\sigma^{d(\mu)+d(\nu)}(\nu x)=\sigma^{d(\mu)+d(\nu)}(\mu(g \cdot  x))=g \vert_{x(0,d(\nu))} \cdot \sigma^{d(\nu)}(x).
\end{align*}
So $d(\mu)=d(\nu)$ and $g \vert_{x(0,d(\nu))}=1_G$. It then follows from  $\mu(g \cdot x)=\nu x$ that $\mu=\nu$ and $g \cdot x=x$.
In particular, $g\cdot x(0,d(\nu))=x(0,d(\nu))$. Since the self-similar action is pseudo free, we deduce that $g=1_G$.

We now suppose that every cycline triple is trivial. Fix $\mu \in \Lambda$, $g \in G$, $p,q \in \mathbb{N}^k$ with  $g \neq 1_G$ or $p \neq q$. In what follows, we verify that
Proposition~\ref{P:Gape} (i)-(ii) are satisfied. We split into two cases.

Case 1. $p \neq q$. To the contrary, assume that, for any $\nu \in s(\mu)\Lambda$ with $d(\nu) \geq p \lor q$, we have $\nu(p,d(\nu)+p-( p \lor q)) = g \vert_{(\mu\nu)(q,d(\mu)+q)} \cdot \nu(q,d(\nu)+q-(p \lor q))$. Then $\sigma^p(x)= g \vert_{(\mu x)(q,d(\mu)+q)} \cdot \sigma^q(x)$ for any $x \in s(\mu)\Lambda^\infty$.
Take an arbitrary $x \in s(\mu)\Lambda^\infty$. We claim that $(g \vert_{(\mu x)(q,d(\mu)+q)} \cdot x(q,p\lor q),g \vert_{(\mu x)(q,d(\mu)+p \lor q)},x(p, p \lor q))$ is a cycline triple.
For this, let $z \in x(p \lor q, p \lor q)\Lambda^\infty$. Then
\begin{align*}
x(p,p \lor q) z&=\sigma^p(x(0,p \lor q)z)
\\&= g \vert_{(\mu x)(q,d(\mu)+q)} \cdot \sigma^q(x(0,p \lor q)z)
\\&=g \vert_{(\mu x)(q,d(\mu)+q)} \cdot (x(q,p \lor q)z)
\\&=(g \vert_{(\mu x)(q,d(\mu)+q)} \cdot x(q,p\lor q))(g \vert_{(\mu x)(q,d(\mu)+p \lor q)} \cdot z).
\end{align*}
Hence $(g \vert_{(\mu x)(q,d(\mu)+q)} \cdot x(q,p\lor q),g \vert_{(\mu x)(q,d(\mu)+p \lor q)},x(p, p \lor q))$ is cycline. By the assumption, one necessarily has
$p=q$, which is a contradiction. Hence Proposition~\ref{P:Gape} (i) holds.

Case 2. $p=q$. To the contrary, assume that there exists $\alpha \in s(\mu)\Lambda$ with $d(\alpha) \geq p$ and $ g \vert_{(\mu\alpha)(p,d(\mu)+p)} \neq 1_G$, such that
$\alpha(p,d(\alpha))\beta = g \vert_{(\mu\alpha)(p,d(\mu)+p)} \cdot (\alpha(p,d(\alpha))\beta)$ for any $\beta \in s(\alpha)\Lambda$. Then we obtain a cycline triple
$(g \vert_{(\mu\alpha)(p,d(\mu)+p)} \cdot \alpha(p,d(\alpha)), g \vert_{(\mu\alpha)(p,d(\mu\alpha))} ,\alpha(p,d(\alpha)))$. By the assumption, we have
\[
g \vert_{(\mu\alpha)(p,d(\mu)+p)} \cdot \alpha(p,d(\alpha))=\alpha(p,d(\alpha)) \text{ and }
g \vert_{(\mu\alpha)(p,d(\mu)+p)} \vert_{\alpha(p,d(\alpha))}=1_G.
\]
So $g \vert_{(\mu\alpha)(p,d(\mu)+p)}=1_G$ as the self-similar action is pseudo free. This yields a contradiction. Therefore Proposition~\ref{P:Gape} (ii) holds.
\end{proof}

Combining Propositions~\ref{P:Gape} and \ref{P:Gapecyc}, one has

\begin{thm}
\label{T:chaGape}
Let $\Lambda$ be a pseudo free self-similar $k$-graph over $G$. Then the following are equivalent.
\begin{enumerate}
\item $\Lambda$ is $G$-aperiodic;
\item for any $\mu \in \Lambda, g \in G, p,q \in \mathbb{N}^k$ with $p \neq q$ or $g \neq 1_G$, Proposition~\ref{P:Gape} (i)-(ii) hold;
\item all cycline triples are trivial.
\end{enumerate}
\end{thm}

\section{The $G$-Periodicity Theory of Self-Similar $k$-Graphs}
\label{S:perthe}

Pseudo free and locally faithful self-similar actions behave extremely well, and form a very interesting class of self-similar actions.
One of most important features of these actions is that their properties mainly inherit from the ambient $k$-graphs.
In this section, we study their $G$-periodicity theory thoroughly. The main results of this section are Theorems~\ref{T:Pergroup} and \ref{T:embed}, which state that, for each such self-similar action $(G,\Lambda)$, its periodicity $\Per_{G,\Lambda}$ forms a subgroup of $\bZ^k$, and the group C*-algebra $\ca(\Per_{G,\Lambda})$ is isomorphic to the $C^*$-subalgebra of $\O_{G,\Lambda}$ generated by the central unitaries $V_{m,n}$ ($m,n\in\bN^k$ with $m-n\in\Per_{G,\Lambda}$).

First notice that, for a locally faithful action $(G,\Lambda)$, if $(\mu,g,\nu)\in\C_{G,\Lambda}$, then $(\mu,\nu)$ is a cycline pair for $\Lambda$ if and only if $g=1_G$.

The following is a simple but key observation, which plays a vital role in the $G$-periodicity theory for locally faithful self-similar actions $(G,\Lambda)$.

\begin{lem}\label{L:unicyc}
Let $\Lambda$ be a locally faithful self-similar $k$-graph over $G$.
\begin{enumerate}
\item
If $(\mu, g_1, \nu_1)$ and $(\mu, g_2, \nu_2)$ are cycline triples
such that $d(\nu_1)=d(\nu_2)$, then $g_1=g_2$ and $\nu_1=\nu_2$.

\item If $(\mu_1', g_1', \nu')$ and $(\mu_2', g_2', \nu')$ are cycline triples such that $d(\mu_1')=d(\mu_2')$,
then $g_1'=g_2'$ and $\mu_1'=\mu_2'$.
\end{enumerate}
\end{lem}

\begin{proof}
The proofs of (i) and (ii) are similar. In what follows, we prove (i) only.
For any $x \in s(\nu_1)\Lambda^\infty$, we have
\[
\nu_2((g_2^{-1}g_1) \cdot x)=\mu (g_2 \cdot ((g_2^{-1}g_1) \cdot x))=\mu(g_1 \cdot x)=\nu_1 x.
\]
Then $\nu_1=\nu_2$ as $d(\nu_1)=d(\nu_2)$, and $(g_2^{-1}g_1) \cdot x=x$ implying $g_1=g_2$ due to
local faithfulness.
\end{proof}

\subsection{Periodicity vs local periodicity vs cycline triples}
Let $(G,\Lambda)$ be a  pseudo free and locally faithful self-similar action with $\Lambda$ strongly connected.
Denote by
\[
\Per_{G,\Lambda}:=\{d(\mu)-d(\nu): (\mu, g, \nu)\in \C_{\Lambda, G}\}
\]
the periodicity from cycline triples.
For $v\in \Lambda^0$, define the local periodicity of $(G,\Lambda)$ at $v$ as
\[
\Sigma^v_{G,\Lambda}:=\{(p,q,g)\in \bN^k\times \bN^k\times G: \sigma^p(x)=\sigma^q(g\cdot x) \text{ for all }x\in v\Lambda^\infty\},
\]
and the periodicity of $(G,\Lambda)$ as
\[
\Sigma_{G,\Lambda}:=\bigcup_{v\in\Lambda^0}\Sigma^v_{G,\Lambda}.
\]
In this subsection, we show that these three types of periodicity can be unified (cf. \cite{DY091, DY09, HLRS15, Yan15}). This
turns out to be very useful for describing the KMS states of $\O_{G,\Lambda}$ in Section \ref{S:KMS}.

The theorem below is a variant of Theorem \ref{T:chaGape} combining with the definition of $\Per_{G,\Lambda}$.

\begin{thm}
\label{T:chaGape1}
Let $\Lambda$ be a pseudo free, locally faithful self-similar $k$-graph over $G$. Then
\begin{center}
$\Lambda$ is $G$-aperiodic $\iff$all cycline triples are trivial$\iff$$\Per_{G,\Lambda}=\{0\}$.
\end{center}
\end{thm}

\begin{lem}\label{global per}
Let $(G,\Lambda)$ be a self-similar action with $\Lambda$ strongly connected, and $p,q\in \bN^k$.
If there is $g\in G$ such that $(p,q,g)\in \Sigma_{G,\Lambda}^v$ for some $v\in\Lambda^0$, then for every $w\in \Lambda^0$, there exists $h\in G$ such that $(p,q,h)\in \Sigma_{G,\Lambda}^w$.
\end{lem}

\begin{proof}
Let $w \in \Lambda^0$. Since $\Lambda$ is strongly connected, there exists $\lambda \in v \Lambda w$ (\cite{HLRS15}). Let $h:=g \vert_\lambda$. Then for any $x \in w \Lambda^\infty$, we have
\[
\sigma^p(x)=\sigma^{p+d(\lambda)}(\lambda x)=\sigma^{d(\lambda)}(\sigma^{q}(g \cdot (\lambda x)))=\sigma^q(h \cdot x).
\]
So we are done.
\end{proof}

\begin{lem}\label{L:1toall}
Let $\Lambda$ be a self-similar action with $\Lambda$ strongly connected, and
let $p,q,p',q' \in \mathbb{N}^k$ with $p-q=p'-q'$. If there is $g\in G$ such that $(p,q,g)\in \Sigma_{G,\Lambda}^v$ for some $v\in \Lambda^0$, then for every $w\in \Lambda^0$, there exists $h\in G$ such that $(p',q',h)\in \Sigma_{G,\Lambda}^w$.
\end{lem}

\begin{proof}
Fix $w \in \Lambda^0$ arbitrarily. Since $\Lambda$ is strongly connected, there exists $\lambda \in \Lambda^p w$ (\cite[Lemma 2.1]{HLRS15}). By Lemma~\ref{global per}, there exists $g' \in G$ such that $(p,q,g')\in\Sigma_{G,\Lambda}^{r(\lambda)}$.
Let $h:=g' \vert_\lambda$. Then for any $x \in w \Lambda^\infty$,
\[
\sigma^{p'}(x)=\sigma^{p'+p}(\lambda x)=\sigma^{p'}(\sigma^{q}(g'\cdot(\lambda x)))=\sigma^{p+q'}(g'\cdot(\lambda x))=\sigma^{q'}(h \cdot x).
\]
So $(p',q',h)\in\Sigma_{G,\Lambda}^w$.
\end{proof}

\begin{prop}\label{P:unicyc}
Let $(G,\Lambda)$ be a locally faithful self-similar action with $\Lambda$ strongly connected.
 Fix $p,q,p',q'$ $\in \mathbb{N}^k$ with $p-q=p'-q'$.
Suppose that $(p,q,g)\in \Sigma_{G,\Lambda}^v$ for some $g\in G$ and $v\in \Lambda^0$.
Then for any $\mu \in \Lambda^{p'}$, there exist unique $h \in G$ and $\nu \in \Lambda^{q'}$ such that $(\mu,h,\nu)$ is a cycline triple; and
similarly for any $\nu' \in \Lambda^{q'}$, there exist unique $h' \in G$ and $\mu' \in \Lambda^{p'}$ such that $(\mu',h',\nu')$ is a cycline triple.
\end{prop}

\begin{proof}
Let $\mu\in\Lambda^{p'}$.
As before, since $\Lambda$ is strongly connected, there exists $\lambda \in \Lambda^{q'}r(\mu)$ (\cite{HLRS15}). By Lemma~\ref{L:1toall}, there exists $g' \in G$, such that
$(p',q',g')\in\Sigma_{G,\Lambda}^{r(\lambda)}$.
Let $\omega:=(\lambda\mu)(0,p')$ and $\nu:=(\lambda\mu)(p'p'+q')$. Then for any $x \in s(\nu)\Lambda^\infty$,
\[
\nu x=\sigma^{p'}(\omega \nu x)=\sigma^{q'}(g'\cdot (\omega\nu x))=\sigma^{q'}(g'\cdot (\lambda\mu x))=g' \vert_\lambda \cdot (\mu x).
\]
So $(\mu,g'\vert_\lambda^{-1} \vert_\nu^{-1},g' \vert_\lambda^{-1} \cdot \nu)$ is a cycline triple. The uniqueness follows from Lemma~\ref{L:unicyc}.

The second statement can be proved similarly.
\end{proof}

\begin{rem}
\label{R:thetapq}
Let $p,q\in \bN^k$ be such that $(p,q,g)\in \Sigma_{G,\Lambda}$ for some $g\in G$. By Proposition \ref{P:unicyc}, there are a bijection $\theta_{p,q}: \Lambda^p\to \Lambda^q$
and a mapping $\phi_{p,q}: \Lambda^p\to G$  such that
$(\mu, \phi_{p,q}(\mu), \theta_{p,q}(\mu))$ is cycline.
When $G$ is trivial, it is easy to see that $\phi_{p,q}$ is trivial and that $\theta_{p,q}$ is nothing but the one defined in \cite[Lemma 5.1]{HLRS15}.
\end{rem}

\begin{prop}
\label{P:2per}
Suppose that $(G,\Lambda)$ is a pseudo free and locally faithful self-similar action with $\Lambda$ strongly connected.
Then the following statements are equivalent:
\begin{enumerate}
\item $\Lambda$ is $G$-periodic;
\item there are $\mu\ne \nu \in \Lambda$ and $g\in G$ such that $(\mu,g,\nu)\in\C_{G,\Lambda}$;
\item there are $m\ne n\in \bN^k$ and $h\in G$ such that $(m,n,h)\in\Sigma_{G,\Lambda}$;
\item for every $v\in\Lambda^0$, there are $m\ne n\in \bN^k$ and $h\in G$ such that $(m,n,h)\in \Sigma_{G,\Lambda}^v$;
\item for every $x\in \Lambda^\infty$, there are $m\ne n\in \bN^k$ and $h\in G$ such that $\sigma^m(x)=\sigma^n(h\cdot x)$;
\item for every $x\in \Lambda^\infty$, there are $m\ne n\in \bN^k$ and $h\in G$ such that $\sigma^m(x)=h\cdot\sigma^n(x)$.
\end{enumerate}
\end{prop}

\begin{proof}
For (i)$\Rightarrow$(ii), notice that if $(\mu, g, \mu)$ is cycline triple, then $g=1_G$ due to the local faithfulness.
So one obtains (ii) by Theorem \ref{T:chaGape1}.
(ii)$\Rightarrow$(iii) and (iii)$\Rightarrow$(iv) are from Lemma \ref{L:equi} and Lemma \ref{L:1toall}, respectively;
(iv)$\Rightarrow$(v) is trivial;
(v)$\Rightarrow$(vi) is due to $\sigma^m(x)=\sigma^n(h\cdot x)=h|_{x(0,n)}\sigma^n(x)$; and
(vi)$\Rightarrow$(i) is from the definition of $G$-aperiodicity of $\Lambda$.
\end{proof}

\subsection{The periodicity group C*-algebra $\ca(\Per_{G,\Lambda})$}
We are now ready to present the main results of this section.

\begin{thm}\label{T:Pergroup}
Suppose that $(G,\Lambda)$ is a locally faithful self-similar action with $\Lambda$ strongly connected.
Then
\begin{align*}
\Per_{G,\Lambda}&=\{p-q: \text{there is } g \in G \text{ such that } (p,q,g)\in \Sigma_{G,\Lambda}\}
\\&=\{p-q: \text{for any } v \in \Lambda^0\, \text{there is } g \in G \text{ such that }(p,q,g)\in\Sigma_{G,\Lambda}^v\}.
\end{align*}
Moreover, $\Per_{G,\Lambda}$ is a subgroup of $\mathbb{Z}^k$.
\end{thm}

\begin{proof}
The first equality easily follows from Lemma \ref{L:equi} and
Proposition~\ref{P:unicyc}, and the second one is from Lemma \ref{global per}.

To see that $\Per_{G,\Lambda}$ is a subgroup of $\mathbb{Z}^k$, it is sufficient to show that $\Per_{G,\Lambda}$ is closed under addition.
For this, let $p,q,m,n \in \mathbb{N}^k$ satisfy that for any $v \in \Lambda^0$, there exist $g_v,h_v \in G$ such that
$(p,q,g_v),(m,n,h_v)\in\Sigma_{G,\Lambda}^v$.
 Fix $v \in \Lambda^0$.
Then for any $x \in v \Lambda^\infty$, we have
\begin{align*}
\sigma^{p+m}(x)
&=\sigma^p(\sigma^{n}(h_v \cdot x))=\sigma^{n}(\sigma^p(h_v \cdot x))\\
&=\sigma^{n}(\sigma^q((g_{h_v \cdot v} h_v)\cdot x))=\sigma^{q+n}((g_{h_v \cdot v}h_v) \cdot x).
\end{align*}
So $(p+m,q+n,g_{h_v \cdot v}h_v)\in \Sigma_{G,\Lambda}^v$.
Therefore $\Per_{G,\Lambda}$ is a subgroup of $\mathbb{Z}^k$.
\end{proof}

For convenience, we write $(\mu,g,\nu)\in \C_{G,\Lambda}^{m,n}$ if $(\mu,g,\nu)\in \C_{G,\Lambda}$ with $d(\mu)=m$ and $d(\nu)=n$.

\begin{thm}\label{T:Vrep}
Let $\Lambda$ be a strongly connected $k$-graph with $\vert \Lambda^0 \vert<\infty$.
Suppose that $(G,\Lambda)$ be a locally faithful self-similar action.
Then for any $m,n\in \mathbb{N}^k$ with $m-n\in \Per_{G,\Lambda}$, the operator
\[
V_{m,n}:=\sum_{(\mu,g,\nu)\in \C_{G,\Lambda}^{m,n}}s_\mu u_g s_\nu^*
\]
is a central unitary in $\O_{G,\Lambda}$, called a {\rm periodicity unitary in $\O_{G,\Lambda}$}.

Moreover, the unitaries $V_{m,n}$ share the following properties:
\begin{itemize}
\item
$V_{\ell,\ell}=1$ for all $\ell\in\bN^k$;
\item
$V_{m,n}^*=V_{n,m}$ for all $m,n\in\bN^k$ with $m-n\in \Per_{G,\Lambda}$;
\item
$V_{p,q}=V_{m,n}$ for all $p,q,m,n\in\bN^k$ with $m-n=p-q\in \Per_{G,\Lambda}$;
\item
$V_{p,q}V_{m,n}=V_{p+m,q+n}$ for all $p,q,m,n\in\bN^k$ with $m-n$, $p-q\in \Per_{G,\Lambda}$.
\end{itemize}
\end{thm}

\begin{proof}
By Proposition~\ref{P:unicyc}, for any $\mu \in \Lambda^m$, there exist unique $g_\mu \in G,\nu_\mu \in \Lambda^n$ such that $(\mu,g_\mu,\nu_\mu)$ is a cycline triple; and for any $\nu \in \Lambda^n$, there exist unique $g_\nu \in G,\mu_\nu \in \Lambda^m$ such that $(\mu_\nu,g_\nu,\nu)$ is a cycline triple. So
\[
V_{m,n}=\sum_{\mu \in \Lambda^m}s_\mu u_{g_\mu} s_{\nu_\mu}^*=\sum_{\nu \in \Lambda^n}s_{\mu_\nu}u_{g_\nu} s_\nu^*.
\]
We compute that
\begin{align*}
V_{m,n}^*V_{m,n}&=\sum_{\nu \in \Lambda^n}s_\nu u_{g_\nu}^* s_{s(\mu_\nu)} u_{g_\nu} s_\nu^*=\sum_{\nu \in \Lambda^n}s_\nu s_\nu^*=1,\\
V_{m,n}V_{m,n}^*&=\sum_{\mu \in \Lambda^m}s_\mu u_{g_\mu} s_{s(\nu_\mu)} u_{g_\mu}^* s_\mu^*=\sum_{\mu \in \Lambda^m}s_\mu s_\mu^*=1.
\end{align*}
So $V_{m,n}$ is a unitary in $\O_{G,\Lambda}$.

It is straightforward to verify that $V_{m,n}^*=V_{n,m}$, $V_{\ell,\ell}=1$, $V_{m+\ell,n+\ell}=V_{m,n}$ for any $\ell \in \mathbb{N}^k$.

In what follows, we show that $V_{m,n}$ is central in $\O_{G,\Lambda}$. First notice that for any $\lambda \in \Lambda$, we have
\begin{align*}
V_{m,n}s_\lambda
&=V_{m+d(\lambda),n+d(\lambda)}s_\lambda\\
&=\mathop{\sum_{\omega \in r(\mu)\Lambda^{d(\lambda)}}}_{(\mu,g,\nu)\in \C_{G,\Lambda}^{m,n}}s_{\omega\mu}u_{g}s_{\omega\nu}^*s_\lambda\\
&=\mathop{\sum_{r(\mu)=s(\lambda)}}_{(\mu,g,\nu)\in \C_{G,\Lambda}^{m,n}}s_{\lambda\mu}u_{g} s_\nu^*\\
& =s_\lambda V_{m,n} \ (\text{by CK1}),
\end{align*}
where the 2ed ``=" used the simple fact that $(\mu',g , \nu')\in\C_{G,\Lambda}$ with $d(\mu')=m + d(\lambda)$ and $d(\nu')=d(\nu)+d(\lambda)$
$\Leftrightarrow \mu'=\omega\mu$, $\nu'=\omega\nu$ with $d(\omega)=d(\lambda)$ and $(\mu, g,\nu) \in \C_{G,\Lambda}$.

Now for any $g \in G$, we calculate that
\begin{align*}
u_gV_{m,n}u_g^*&=\sum_{(\mu,h,\nu)\in\C_{G,\Lambda}^{m,n}}s_{g \cdot \mu} u_{g\vert_\mu h(g\vert_\nu)^{-1}} s_{g \cdot \nu}^*=V_{m,n}
\end{align*}
as $(\mu, h, \nu)\in\C_{G,\Lambda}\iff (g \cdot \mu,g\vert_\mu h( g\vert_\nu)^{-1},g \cdot \nu) \in \C_{G,\Lambda}$.

Therefore, $V_{m,n}$ is central in $\O_{G,\Lambda}$.

It remains to show the last property of $V_{m,n}$. For this, let $p,q,m,n\in\bN^k$ with $m-n$, $p-q\in \Per_{G,\Lambda}$.
We compute that
\begin{align*}
&V_{p,q}V_{m,n}\\
&=\sum_{(\mu,g,\nu)\in \C_{G,\Lambda}^{p,q}}
      \sum_{(\alpha,h,\beta)\in \C_{G,\Lambda}^{m,n}}s_\mu u_g s_\nu^*s_\alpha u_h s_\beta^*\\
&=\sum_{(\mu,g,\nu)\in \C_{G,\Lambda}^{p,q}}
     \sum_{(\alpha,h,\beta) \in \C_{G,\Lambda}^{m,n}}
     \mathop{\sum_{\nu\lambda=\alpha\omega}}_{\lambda \in \Lambda^m, \omega\in\Lambda^q}\,
      s_{\mu (g\cdot \lambda)}u_{g \vert_\lambda h \vert_{h^{-1} \cdot \omega}}s_{\beta (h^{-1} \cdot \omega)}^*\\
&=V_{p+m,q+n}
\end{align*}
as  $(\mu (g\cdot \lambda),g \vert_\lambda h \vert_{h^{-1} \cdot \omega} ,\beta (h^{-1} \cdot \omega)) \in \C_{G,\Lambda}^{p+m,q+n}$ and $V_{p,q}V_{m,n}$ is unitary.
\end{proof}

\begin{thm}
\label{T:embed}
Let $(G,\Lambda)$ be a  pseudo free and locally faithful self-similar action with $\Lambda$ strongly connected and $\vert \Lambda^0 \vert<\infty$.
Then
\[
\ca(\Per_{G,\Lambda})\cong \ca(\{V_{m,n}: m,n\in\bN^k, \, m-n\in\Per_{G,\Lambda}\}).
\]
\end{thm}

\begin{proof}
Let $\fA:= \ca(\{V_{m,n}: m,n\in\bN^k, \, m-n\in\Per_{G,\Lambda}\})$.
Denote by $\{v_z:z \in \Per_{G,\Lambda}\}$ the set of generators of $\ca(\Per_{G,\Lambda})$. By the universal property of $\ca(\Per_{G,\Lambda})$
and Theorem \ref{T:Vrep}, there exists a homomorphism
$\iota:\ca(\Per_{G,\Lambda}) \to \fA$
such that $\iota(v_{p-q})=V_{p,q}$ whenever $p,q\in \bN^k$ with $p-q \in \Per_{G,\Lambda}$. By \cite[Proposition~4.1.9]{BO08}, there exists a faithful expectation 
(actually, a faithful tracial state) $E_1:\ca(\Per_{G,\Lambda}) \to \mathbb{C}$ such that $E_1(v_0)=1$ and $E_1(v_z)=0$ for all $0\ne z \in \Per_{G,\Lambda}$. As shown in the proof of \cite[Theorem~6.13]{LY17_2}, there exists an expectation $E_2:\mathcal{O}_{G,\Lambda} \to \fA$ such that for any $\mu,\nu \in \Lambda,g \in G$ with $s(\mu)=g \cdot s(\nu)$,
\begin{align*}
E_2(s_\mu u_g s_\nu^*)= \begin{cases}
    s_\mu s_\mu^*  &\text{ if $\mu=\nu \text{ and }g=1_G$} \\
    0 &\text{ otherwise}.
\end{cases}
\end{align*}
So the restriction of $E_2$ to $\fA$ is also an expectation. Since $\iota \circ E_1=E_2 \circ \iota$, by \cite[Proposition~3.11]{Kat03} $\iota$ is injective.
\end{proof}


A simple but useful observation about periodicity unitaries is the following identities: If $(\mu,g,\nu)\in \C_{G,\Lambda}^{m,n}$, then
\begin{align}
\label{E:Vmn}
s_\mu u_g s_\nu^*= s_\mu s_\mu^* V_{m,n}=V_{m,n}s_\nu s_\nu^*.
\end{align}

\section{Cartan Subalgebras of $\O_{G,\Lambda}$}
\label{S:Cartan}

Let $(G,\Lambda)$ be a pseudo free and locally faithful self-similar action with $\Lambda$ strongly connected and $\vert \Lambda^0 \vert<\infty$. In this section, we will find a canonical Cartan subalgebra for the self-similar $k$-graph C*-algebra $\O_{G,\Lambda}$ by invoking the very useful criterion given by Brown-Nagy-Reznikoff-Sims-Williams in \cite{BNRSW16}. To be more precise, we show that the self-similar cycline subalgebra $\M_{G,\Lambda}$ of $\O_{G,\Lambda}$ is Cartan (Theorem \ref{T:Exp}), which turns out to be the tensor product of the group C*-algebra $\ca(\Per_{G,\Lambda})$ with the diagonal subalgebra $\D_\Lambda$ (Theorem \ref{T:tensor}). These results will play a central role in computing the KMS states of 
$\O_{G,\Lambda}$ later.

\begin{defn}
\label{D:Mcyc}
Let $\Lambda$ be a self-similar $k$-graph over a group $G$  with $\vert \Lambda^0 \vert<\infty$.
The \emph{self-similar cycline subalgebra} $\M_{G,\Lambda}$ is defined as the C*-subalgebra of
    $\O_{G,\Lambda}$ generated by $s_\mu u_g s_\nu^*$ with $(\mu, g, \nu)$  cycline:
\[
\M_{G,\Lambda}:=\ca(\{s_\mu u_g s_\nu^*: (\mu, g, \nu)\in\C_{G,\Lambda}\}).
\]
\end{defn}

\begin{lem}\label{L:Isoint}
Let $\Lambda$ be a pseudo free self-similar $k$-graph over a group $G$, and
$\mathrm{Iso}(\mathcal{G}_{G,\Lambda})^{\mathrm{o}}$ be the interior of the isotropy of $\mathcal{G}_{G,\Lambda}$. Then
$\mathrm{Iso}(\mathcal{G}_{G,\Lambda})^{\mathrm{o}}=\{(\mu (g \cdot x);\T_{d(\mu)}([g \vert_{x}]),d(\mu)-d(\nu);\nu x):(\mu,g,\nu) \in \C_{G,\Lambda},\, x \in s(\nu)\Lambda^\infty\}$.
\end{lem}

\begin{proof}
We notice that
\[
\mathrm{Iso}(\mathcal{G}_{G,\Lambda})^{\mathrm{o}}
=\bigcup_{Z(\mu,g,\nu) \subseteq \mathrm{Iso}(\mathcal{G}_{G,\Lambda})}Z(\mu,g,\nu)
=\bigcup_{(\mu,g,\nu) \in \C_{G,\Lambda}}Z(\mu,g,\nu).
\]
So we are done.
\end{proof}

\begin{lem}
\label{L:Isoabelian}
Let $\Lambda$ be a pseudo free and locally faithful self-similar $k$-graph over a group $G$.
Then $\mathrm{Iso}(\mathcal{G}_{G,\Lambda})^{\mathrm{o}}$ is abelian.
\end{lem}

\begin{proof}
Let  $(\mu,g,\nu)$ and $(\alpha,h,\beta)$ be two cycline triples, $x \in s(\nu)\Lambda^\infty$, and $y \in s(\beta) \Lambda^\infty$.
Notice that $\nu x=\alpha (h \cdot y)$ if and only if $\beta y=\mu (g \cdot x)$.
Suppose that $\nu x=\alpha (h \cdot y)$. So in order to prove that
$(\mu (g \cdot x);\T_{d(\mu)}([g \vert_{x}]),d(\mu)-d(\nu);\nu x)$ and $(\alpha (h \cdot y);\T_{d(\alpha)}([h \vert_{y}]),d(\alpha)-d(\beta);\beta y)$ commute, it is sufficient to show that
\[
\T_{d(\mu)}([g \vert_{x}])\T_{d(\mu)-d(\nu)+d(\alpha)}([h \vert_{y}])
   =\T_{d(\alpha)}([h \vert_{y}])\T_{d(\alpha)-d(\beta)+d(\mu)}([g \vert_{x}]),
\]
which is equivalent to
\begin{align*}
&\T_{d(\mu)+d(\alpha)}([g \vert_{x(0,d(\alpha)}|_{\sigma^{d(\alpha)}(x)}])\T_{d(\mu)-d(\nu)+d(\alpha)}\T_{d(\nu)}([h \vert_{x(0,d(\nu)}|_{\sigma^{d(\nu)}(y)}])\\
&=\T_{d(\alpha)+d(\mu)}([h \vert_{y(0,d(\mu)}|_{\sigma^{d(\mu)}(y)}])\T_{d(\alpha)-d(\beta)+d(\mu)}\T_{d(\beta)}([g \vert_{x(0,d(\beta)}|_{\sigma^{d(\beta)}(x)}])
\end{align*}
by \cite[Lemma 5.1]{LY17_2}. Thus one has to show
\begin{align}
\label{E:Iso}
[g \vert_{x(0,d(\alpha))} \vert_{\sigma^{d(\alpha)}(x)}  h \vert_{y(0,d(\nu))} \vert_{\sigma^{d(\nu)}(y)}]
   =[h \vert_{y(0,d(\mu))} \vert_{\sigma^{d(\mu)}(y)} g \vert_{x(0,d(\beta))} \vert_{\sigma^{d(\beta)}(x)}].
\end{align}
To do so, let $p \in \mathbb{N}^k$. For any $z \in y(d(\nu)+p,d(\nu)+p)\Lambda^\infty$, we compute that
\begin{align*}
&\mu (g \cdot x(0,d(\alpha)+p)) ((g \vert_{x(0,d(\alpha)+p)}h \vert_{y(0,d(\nu)+p)}) \cdot z)
\\&=\mu (g \cdot (x(0,d(\alpha)+p)(h \vert_{y(0,d(\nu)+p)} \cdot z)))
\\&=\nu x(0,d(\alpha)+p)(h \vert_{y(0,d(\nu)+p)} \cdot z)
\\&=\alpha (h \cdot y(0,d(\nu)+p)) (h \vert_{y(0,d(\nu)+p)} \cdot z)
\\&=\alpha (h \cdot (y(0,d(\nu)+p)z))
\\&=\beta y(0,d(\nu)+p)z.
\end{align*}
So $(\mu (g \cdot x(0,d(\alpha)+p)),g \vert_{x(0,d(\alpha)+p)}h \vert_{y(0,d(\nu)+p)} ,\beta y(0,d(\nu)+p))$ is a cycline triple. On the other hand, for any $z \in x(d(\beta)+p,d(\beta)+p) \Lambda^\infty$, we calculate that
\begin{align*}
&\alpha (h \cdot y(0,d(\mu)+p)) ((h \vert_{y(0,d(\mu)+p)} g \vert_{x(0,d(\beta)+p)})\cdot z)
\\&=\alpha (h \cdot (y(0,d(\mu)+p)(g \vert_{x(0,d(\beta)+p)} \cdot z)))
\\&=\beta y(0,d(\mu)+p)(g \vert_{x(0,d(\beta)+p)} \cdot z)
\\&=\mu (g \cdot x(0,d(\beta)+p)) (g \vert_{x(0,d(\beta)+p)} \cdot z)
\\&=\mu (g \cdot (x(0,d(\beta)+p) z))
\\&=\nu x(0,d(\beta)+p) z.
\end{align*}
So $(\alpha (h \cdot y(0,d(\mu)+p)),h \vert_{y(0,d(\mu)+p)}g \vert_{x(0,d(\beta)+p)} ,\nu x(0,d(\beta)+p) )$ is a cycline triple. Since
$\mu g\cdot x=\nu x=\alpha(h\cdot y)$, one has
$
\mu (g \cdot x(0,d(\alpha)+p))=\alpha (h \cdot y(0,d(\mu)+p)).
$
It follows from Lemma~\ref{L:unicyc} that \eqref{E:Iso} holds true.
Hence $\mathrm{Iso}(\mathcal{G}_{G,\Lambda})^{\mathrm{o}}$ is abelian.
\end{proof}


The following theorem is a generalization of \cite[Theorem~4.4]{BLY17}.

\begin{prop}\label{P:Cartan}
Let $\Lambda$ be a pseudo free self-similar $k$-graph over a group $G$. Then $\mathrm{Iso}(\mathcal{G}_{G,\Lambda})^{\mathrm{o}}$ is closed
if and only if there exist no $\mu,\nu \in \Lambda,g \in G,x \in s(\nu) \Lambda^\infty$ with $s(\mu)=g \cdot s(\nu)$ satisfying that
\begin{enumerate}
\item[(1)] 
for $p \in \mathbb{N}^k,(\mu (g \cdot x(0,p)),g\vert_{x(0,p)},\nu x(0,p))$ is not a cycline triple;
\item[(2)] 
for $p \in \mathbb{N}^k$, there exists $\lambda \in x(p,p)\Lambda$ such that $(\mu (g \cdot (x(0,p)\lambda))$, $g\vert_{x(0,p)\lambda},\nu x(0,p)\lambda)$ is a cycline triple.
\end{enumerate}
\end{prop}

\begin{proof}
Firstly, suppose that $\mathrm{Iso}(\mathcal{G}_{G,\Lambda})^{\mathrm{o}}$ is closed. Suppose that there exist $\mu,\nu \in \Lambda,g \in G,x \in s(\nu) \Lambda^\infty$ with $s(\mu)=g \cdot s(\nu)$ satisfying Conditions~(1)-(2). By Lemma \ref{L:Isoint}, Condition~(1) implies that $(\mu (g \cdot x);\T_{d(\mu)}([g \vert_{x}]),d(\mu)$ $-d(\nu);\nu x) \notin \mathrm{Iso}(\mathcal{G}_{G,\Lambda})^{\mathrm{o}}$, and
Condition~(2) yields that $(\mu (g \cdot x);\T_{d(\mu)}([g \vert_{x}]),d$ $(\mu)-d(\nu);\nu x) \in \overline{\mathrm{Iso}(\mathcal{G}_{G,\Lambda})^{\mathrm{o}}}$. So $\mathrm{Iso}(\mathcal{G}_{G,\Lambda})^{\mathrm{o}}$ is not closed, which is a contradiction. Hence there exist no $\mu,\nu \in \Lambda,g \in G,x \in s(\nu) \Lambda^\infty$ with $s(\mu)=g \cdot s(\nu)$ satisfying Conditions~(1)-(2).

We now prove the converse statement. Fix $\mu,\nu \in \Lambda,g \in G,x_n,x \in s(\nu) \Lambda^\infty$ with $s(\mu)=g \cdot s(\nu)$ and $x_n \to x$ as $n\to \infty$, such that
 $(\mu (g \cdot x_n);\T_{d(\mu)}([g \vert_{x_n}]),d(\mu)-d(\nu);\nu x_n) \in \mathrm{Iso}(\mathcal{G}_{G,\Lambda})^{\mathrm{o}}$ for all $n \geq 1$. To the contrary, assume that
 $(\mu (g \cdot x);\T_{d(\mu)}([g \vert_{x}]),d(\mu)$ $-d(\nu);\nu x) \notin \mathrm{Iso}(\mathcal{G}_{G,\Lambda})^{\mathrm{o}}$. By Lemma~\ref{L:Isoint}, we deduce that $\mu,\nu,g,x$ satisfy Conditions~(1)-(2), which is a contradiction. So $(\mu (g \cdot x);\T_{d(\mu)}([g \vert_{x}]),d(\mu)$ $-d(\nu);\nu x) \in \mathrm{Iso}(\mathcal{G}_{G,\Lambda})^{\mathrm{o}}$. Hence $\mathrm{Iso}(\mathcal{G}_{G,\Lambda})^{\mathrm{o}}$ is closed.
\end{proof}

\begin{prop}
\label{P:CycCartan}
Let $(G,\Lambda)$ be a pseudo free and locally faithful self-similar action with $\Lambda$ strongly connected.
Then $\mathrm{Iso}(\mathcal{G}_{G,\Lambda})^{\mathrm{o}}$ is closed.
\end{prop}

\begin{proof}
To the contrary, assume that $\Iso(\G_{G,\Lambda})^{\mathrm{o}}$ is not closed.
Then there exist $\mu,\nu \in \Lambda,g \in G,x \in s(\nu) \Lambda^\infty$ with $s(\mu)=g \cdot s(\nu)$ satisfying Conditions~(1)-(2)
of Proposition~\ref{P:Cartan}.
Then there is $p\in \bN^k$ such that $(\mu (g \cdot x(0,p)),g\vert_{x(0,p)},\nu x(0,p))\not\in\C_{G,\Lambda}$,
but there exists $\lambda \in x(p,p)\Lambda$ such that $(\mu (g \cdot (x(0,p)\lambda)), g\vert_{x(0,p)\lambda},\nu x(0,p)\lambda)\in \C_{G,\Lambda}$.
Clearly, $d(\mu (g \cdot (x(0,p)\lambda)))-d(\nu x(0,p)\lambda)=d(\mu (g \cdot x(0,p)))-d(\nu x(0,p))$. By Lemma \ref{L:equi} and Proposition \ref{P:unicyc},
there are unique $h\in G$ and $\omega\in \Lambda^{d(\nu)+p}$ such that $(\mu (g \cdot x(0,p)),h,\omega)\in \C_{G,\Lambda}$.
One can check that
$(\mu g\cdot (x(0,p)\lambda), {h|_{h^{-1}g|_{x(0,p)}\cdot \lambda}}, \omega h^{-1}g|_{x(0,p)}\cdot \lambda)\in\C_{G,\Lambda}$.
It then follows from Lemma \ref{L:unicyc} that
\[
\omega h^{-1}g|_{x(0,p)}\cdot \lambda=\nu x(0,p)\lambda \quad\text{and}\quad h|_{h^{-1}g|_{x(0,p)\lambda}}=g|_{x(0,p)\lambda}.
\]
Thus $\omega=\nu x(0,p)$ and $g|_{x(0,p)}\cdot \lambda=h\cdot\lambda$. Then we have $h|_\lambda=g|_{x(0,p)}|_{\lambda}$. The
pseudo-freeness gives $h=g|_{x(0,p)}$, and so $(\mu (g \cdot x(0,p)),g\vert_{x(0,p)},\nu x(0,p))\in \C_{G,\Lambda}$. This
is a contradiction. By Proposition~\ref{P:Cartan}, $\mathrm{Iso}(\mathcal{G}_{G,\Lambda})^{\mathrm{o}}$ is closed.
\end{proof}

\begin{thm}
\label{T:Exp}
Suppose that $\Lambda$ is a strongly connected $k$-graph with  $\vert \Lambda^0 \vert<\infty$, and that
$G$ is amenable. Let  $(G,\Lambda)$ be a pseudo free and locally faithful self-similar action. Then
\begin{enumerate}
\item $\mathcal{M}_{G,\Lambda}$ is a Cartan subalgebra of $\mathcal{O}_{G,\Lambda}$;

\item there exists a faithful expectation $E:\mathcal{O}_{G,\Lambda} \to \mathcal{M}_{G,\Lambda}$ such that $E(s_\mu u_g s_\nu^*)=0$ if $(\mu,g,\nu)\not\in \C_{G,\Lambda}$.
\end{enumerate}
\end{thm}

\begin{proof}
(i) By Lemma \ref{L:Isoabelian}, Proposition \ref{P:CycCartan}, \cite[Corollary~4.5]{BNRSW16}, $\ca(\mathrm{Iso}(\mathcal{G}_{G,\Lambda})^{\mathrm{o}})$ is a Cartan subalgebra of $\ca(\G_{G,\Lambda})$. Then combining this with \cite[Theorem~5.9]{LY17_2} proves (i).

(ii) By (i) above, there exists a faithful expectation $E:\ca(\mathcal{G}_{G,\Lambda}) \to \ca(\mathrm{Iso}(\mathcal{G}_{G,\Lambda})^{\mathrm{o}})$ such that $E(f)=f \vert_{\mathrm{Iso}(\mathcal{G}_{G,\Lambda})^{\mathrm{o}}}$ for all $f \in C_c(\mathcal{G}_{G,\Lambda})$. Let
$\mu,\nu\in\Lambda$ and $g \in G$ with $s(\mu)=g \cdot s(\nu)$, but
$(\mu,g,\nu)\not\in\C_{G,\Lambda}$. We claim that $Z(\mu,g,\nu) \cap \mathrm{Iso}(\mathcal{G}_{G,\Lambda})^{\mathrm{o}}=\mt$.
To the contrary, assume that $Z(\mu,g,\nu) \cap \mathrm{Iso}(\mathcal{G}_{G,\Lambda})^{\mathrm{o}}\neq\mt$. By Lemma~\ref{L:Isoint} and the pseudo freeness, there exists $\lambda \in s(\nu)\Lambda$ such that $(\mu(g \cdot \lambda),g\vert_\lambda,\nu\lambda)\in \C_{G,\Lambda}$. By  Theorem~\ref{T:Pergroup}, one has
$d(\mu)-d(\nu) \in \Per_{G,\Lambda}$. Then by Proposition~\ref{P:unicyc}, there exist $\omega \in \Lambda^{d(\mu)}$ and $h \in G$ such that
$(\omega,h,\nu)\in\C_{G,\Lambda}$. So $(\omega (h \cdot \lambda),h \vert_\lambda,\nu\lambda)\in\C_{G,\Lambda}$. By Lemma~\ref{L:unicyc},
we have that $\mu g\cdot \lambda=\omega h\cdot \lambda$, implying $\mu=\omega$ and $g\cdot \lambda= h\cdot \lambda$,  and that $g|_\lambda=h|_\lambda$.
The pseudo freeness further yields $g=h$. Thus, $(\mu,g,\nu)\in\C_{G,\Lambda}$, a contradiction. Therefore,
$Z(\mu,g,\nu) \cap \mathrm{Iso}(\mathcal{G}_{G,\Lambda})^{\mathrm{o}}=\mt$. It follows from Lemma \ref{L:Isoint} and \cite[Theorem 5.9]{LY17_2} that $E(s_\mu u_g s_\nu^*)=0$ if $(\mu,g,\nu)\not\in\C_{G,\Lambda}$.
\end{proof}

\begin{thm}
\label{T:tensor}
Suppose that $\Lambda$ is a strongly connected $k$-graph with  $\vert \Lambda^0 \vert<\infty$, and that
$G$ is amenable. Let  $(G,\Lambda)$ be a pseudo free and locally faithful self-similar action.
 Then $\mathcal{M}_{G,\Lambda}\cong\ca(\Per_{G,\Lambda}) \otimes \D_\Lambda$.
\end{thm}

\begin{proof}
Let $\mathcal{F}:=\ca(s_\mu s_\nu^*:d(\mu)=d(\nu))$ and
$\mathcal{B}:=\ca(\{V_{p,q},s_\mu s_\nu^*:p,q\in\bN^k \text{ with } p-q \in \Per_{G,\Lambda}, \mu,\nu\in\Lambda\text{ with }d(\mu)=d(\nu)\})$. Then $\mathcal{F} \subseteq\mathcal{B}$. By Theorem~\ref{T:Vrep},
\[
\ca(\Per_{G,\Lambda})\cong \ca(\{V_{p,q}:p,q\in\bN^k \text{ with } p-q \in \Per_{G,\Lambda}\})\hookrightarrow\B.
\]
So there exists a surjective homomorphism $f:\ca(\Per_{G,\Lambda}) \otimes \F \to \mathcal{B}$ by \cite[Proposition 3.3.7]{BO08}. We claim that $f$ is also injective.
To prove this, for $v \in \Lambda^0$ and $p \in \mathbb{N}^k$, let $\mathcal{F}_p(v):=\ca(s_\mu s_\nu^*:\mu,\nu \in \Lambda^p v)$ and
 $Q:=\sum_{\mu \in \Lambda^p v}s_\mu s_\mu^*$. By \cite[Lemma~3.2]{KP00}, $\mathcal{F} \cong \lim_{p \in \mathbb{N}^k}(\oplus_{v \in \Lambda^0}\mathcal{F}_{p}(v))$.
 A similar argument used in the proof of \cite[Lemma~8.1]{DY09} yields that the restriction $f|_{\ca(\Per_{G,\Lambda}) \otimes \mathcal{F}_p(v)}$ is an isomorphism
 from $\ca(\Per_{G,\Lambda}) \otimes \mathcal{F}_p(v)$ onto $Q \ca(V_{p,q},s_\mu s_\nu^*:p,q\in\bN^k\text{ with }p-q \in \Per_{G,\Lambda}, \mu,\nu \in \Lambda^p v)Q$. So $f$ is injective. Therefore
 $f|_{\ca(\Per_{G,\Lambda}) \otimes \D_\Lambda}$ induces an isomorphism from $\ca(\Per_{G,\Lambda}) \otimes \D_\Lambda$ onto $\mathcal{M}_{G,\Lambda}$.
\end{proof}

\section{KMS States of $\O_{G,\Lambda}$}
\label{S:KMS}

Let $G$ be an amenable group, and $\Lambda$ a strongly connected finite $k$-graph. Suppose that $(G,\Lambda)$ is a pseudo free and locally faithful self-similar action which satisfies the finite-state condition.
In this section,
we will achieve our main goal of this paper -- characterizing the KMS simplex structure of the self-similar $k$-graph C*-algebra $\O_{G,\Lambda}$ (Theorem \ref{T:KMSPer}).
It turns out that the KMS simplex of $\O_{G,\Lambda}$ is affinely isomorphic to the tracial state space of the C*-algebra of the periodicity group $\Per_{G,\Lambda}$
if the Perron-Frobenius eigenvector of $\Lambda$ preserves the $G$-action on $\Lambda^0$, and empty otherwise. 
This generalizes the main result of \cite{HLRS15} for strongly connected finite $k$-graphs C*-algebras (i.e., when $G$ is trivial).

\subsection{Basic properties}
\label{SS:basicKMS}

In this subsection, we recall the notion of KMS states from \cite{BR97} (also see \cite{CM08, Ped79}), and give some basic properties of KMS states for $\O_{G,\Lambda}$.

\begin{defn}
Let $A$ be a C*-algebra, $\alpha$ be an action of $\bR$ on $A$,
and $A^a$ be the set of all analytic elements of $A$. Let  $0<\beta<\infty$. A state $\tau$ of $A$ is called a \emph{KMS$_\beta$ state} of $(A,\mathbb{R},\alpha)$ if $\tau(xy)=\tau(y \alpha_{i\beta}(x))$ for all $x,y \in A^a$.
\end{defn}

Let $\gamma:\prod_{i=1}^k\bT \to \Aut(\mathcal{O}_{G,\Lambda})$ be the strongly continuous homomorphism such that
\[
\gamma_z(s_\mu)=z^{d(\mu)}s_\mu\text{ and } \gamma_z(u_g)=u_g\qforal  z \in \prod_{i=1}^k\bT, \mu \in \Lambda,g \in G.
\]
Let $r\in \prod_{i=1}^k\bR$. Define a strongly continuous homomorphism $\alpha^r:\mathbb{R} \to \Aut(\mathcal{O}_{G,\Lambda})$ by $\alpha^r_t:=\gamma_{e^{itr}}$. Notice that, for $\mu,\nu \in \Lambda,g \in G$ with $s(\mu)=g\cdot s(\nu)$, the function $\mathbb{C} \to \mathcal{O}_{G,\Lambda}$, $\xi \mapsto e^{i \xi r \cdot (d(\mu)-d(\nu))}s_\mu u_g s_\nu^*$ is an entire function.
So $s_\mu u_g s_\nu^*$ is an analytic element. By Proposition~\ref{P:genO}, in order to check the KMS$_\beta$ condition, it is sufficient to check whether it is valid on the set $\{s_\mu u_g s_\nu^*:\mu,\nu \in \Lambda,g \in G,s(\mu)=g\cdot s(\nu)\}$.
In this subsection, we study basic properties of KMS$_\beta$ states of the one-parameter dynamical system $(\mathcal{O}_{G,\Lambda},\mathbb{R},\alpha^r)$.

Recall from Proposition \ref{P:genO} that $\mathcal{O}_\Lambda$ naturally embeds into $\mathcal{O}_{G,\Lambda}$. So we regard $\O_\Lambda$ as a C*-subalgebra of $\O_{G,\Lambda}$ below.

\begin{prop}
\label{P:phiLambda}
Let $r\in \prod_{i=1}^k\bR$. Suppose that $\phi$ is a KMS$_\beta$ state of $(\mathcal{O}_{G,\Lambda},\mathbb{R},\alpha^r)$. Then
\begin{enumerate}
\item $\beta r=\ln (\rho(\Lambda))$;
\item the restriction of $\phi|_{\D_\Lambda}=\phi_\Lambda$, the Perron-Frobenius state $\phi_\Lambda$ of $\Lambda$;
\item $\phi(s_{g \cdot v})=\phi(s_v)$ for all $v \in \Lambda^0,g \in G$.
\end{enumerate}
\end{prop}

\begin{proof}
For $v \in \Lambda^0, 1\leq i \leq k$, we compute that
\begin{align*}
\phi(s_v)=\sum_{\mu \in v \Lambda^{e_i}}\phi(s_\mu s_\mu^*)=e^{-\beta r_i}\sum_{\mu \in v\Lambda^{e_i}}\phi(s_{s(\mu)})=e^{-\beta r_i}\sum_{w \in \Lambda^0}T_{e_i}(v,w)\phi(s_w).
\end{align*}
Let $y:=(\phi(s_v))_{v \in \Lambda^0}$. It follows that for $1 \leq i \leq k,T_{e_i}y=e^{\beta r_i}y$. By Theorem~\ref{T:PF}, $\beta r=\ln (\rho(\Lambda))$ and $x_\Lambda=y$. So for $\mu \in \Lambda$, we have
\[
\phi(s_\mu s_\mu^*)=e^{-\beta r d(\mu)}\phi(s_{s(\mu)})=\rho(\Lambda)^{-d(\mu)}x_{\Lambda}(s_\mu).
\]
By Theorem~\ref{T:PF}, one has $\phi \vert_{\D_\Lambda}=\phi_\Lambda$. This proves (i) and (ii) simultaneously.

Finally, for $v \in \Lambda^0,g \in G$, we compute that $\phi(s_{g \cdot v})=\phi(s_{g \cdot v}u_{g\vert_v}u_g^*)=\phi(u_g s_v u_g^*)=\phi(s_v)$, proving (iii).
\end{proof}

\begin{rem}
\label{R:KMSass}
Proposition \ref{P:phiLambda} (i) says that whenever $\beta \in (0,\infty)$ is given, $r$ is uniquely determined in order to satisfy the KMS$_\beta$ condition for
$(\mathcal{O}_{G,\Lambda},\mathbb{R},\alpha^r)$.
Also, to ensure that the KMS$_\beta$ simplex of $(\mathcal{O}_{G,\Lambda},\mathbb{R},\alpha^r)$ is nonempty,
it follows from Theorem \ref{T:PF} and Proposition \ref{P:phiLambda} (ii)-(iii) that 
 the Perron-Frobenuous eigenvector $x_\Lambda$ of $\Lambda$
has to preserve the action of $G$ on $\Lambda^0$: 
$x_\Lambda(g \cdot v)=x_\Lambda(v)$ for all $v \in \Lambda^0$ and $g \in G$. 
Furthermore, for any $\beta_1,\beta_2 \in (0,\infty)$, the KMS$_{\beta_1}$ simplex of $(\mathcal{O}_{G, \Lambda},\mathbb{R},\alpha^{r_1})$ is affinely
isomorphic to the KMS$_{\beta_2}$ simplex of $(\mathcal{O}_{G, \Lambda},\mathbb{R},\alpha^{r_2})$ (also refer to \cite[Section 7]{HLRS15}).
\end{rem}

Therefore,  \textsf{throughout the rest of this paper, we assume that
 \begin{align}
 \nonumber
\beta&=1 \text{ (and so }r=\ln (\rho(\Lambda))),\\
 \label{E:KMSass}
x_\Lambda(g \cdot v)&=x_\Lambda(v)\qforal v \in \Lambda^0,g \in G.
 \end{align}
 }
That is, we compute the KMS$_1$ states of the preferred dynamical system $(\O_{G,\Lambda}, \mathbb{R},\alpha^{\ln(\rho(\Lambda))})$. As mentioned in the introduction
section, these
KMS states are simply called the \textit{KMS states of $\O_{G,\Lambda}$}.



\subsection{Auxiliary results} Some auxiliary results will be given in this subsection in order to compute the KMS states of $\O_{G,\Lambda}$.
This subsection is heavily motivated from \cite{HLRS15}.

In what follows, $\Lambda$ is a strongly connected finite $k$-graph, and $(G,\Lambda)$ is a pseudo free locally faithful self-similar action satisfying the finite-state condition.

Recall that, for $\mu,\nu \in \Lambda$, let
\[
\Lambda^{\min}(\mu,\nu):=\{(\alpha,\beta) \in \Lambda \times \Lambda:\mu\alpha=\nu\beta,d(\mu\alpha)=d(\mu)\lor d(\nu)\}.
\]

\begin{lem}
\label{L:min}
Let $(m,n,g)\in \bN^k\times \bN^k\times G$ and $v\in\Lambda^0$. If $(m,n,g)\not\in\Sigma^v_{G,\Lambda}$, then there is $\lambda\in v(\Lambda \setminus \Lambda^0)$ such that
$\Lambda^{\min} (\mu (g \cdot \lambda), \nu \lambda)=\mt$ for all $\mu,\nu\in\Lambda$ with $s(\mu)=g\cdot s(\nu)$, $s(\nu)=v$, and $d(\mu)-d(\nu) = m-n$.
\end{lem}

\begin{proof}
The proof is completely similar to \cite[Lemma 8.3]{HLRS15}.

Let $m'=(m-n)\vee 0$ and $n'=(n-m)\vee 0$. Then $m=m'+p$ and $n=n'+p$ for some $p\in \bN^k$. Clearly, $(m',n',g)\not\in\Sigma^v_{G,\Lambda}$
as $(m,n,g)\not\in\Sigma^v_{G,\Lambda}$.
Then there is $x\in v\Lambda^\infty$ such that
$\sigma^{m'}(x)\ne \sigma^{n'}(g\cdot x)$.
So there is $0\ne \ell \in \bN^k$ such that $\sigma^{m'}(x)(0,\ell)\ne \sigma^{n'}(g\cdot x)(0,\ell)$.
Let
\[
\lambda:=x(0,m'+n'+\ell).
\]

Let $p'\in\bN^k$ be such that $d(\mu)=m'+p'$ and $d(\nu)=n'+p'$. Then factor $\mu$ and $\nu$ as $\mu=\alpha\mu'$ and $\nu=\beta\nu'$ with $d(\alpha)=d(\beta)=p'$.

If $\alpha\ne \beta$, then $\Lambda^{\min}(\mu,\nu)=\mt$ and so $\Lambda^{\min}(\mu g\cdot \lambda,\nu\lambda)=\mt$.

If $\alpha=\beta$, then $\Lambda^{\min}(\mu g\cdot \lambda,\nu\lambda)=\Lambda^{\min}(\mu' g\cdot \lambda,\nu'\lambda)$.
But
\begin{align*}
(\mu'g\cdot \lambda) (m'+n',a)
&=(g\cdot \lambda)(n',n'+\ell)=(g\cdot x)(n',n'+\ell)\\
&=\sigma^{n'}(g\cdot x)(0,\ell)\\
&\ne \sigma^{m'}( x)(0,\ell)\\
&= x(m',m'+\ell)=\lambda(m',m'+\ell)\\
&=(\nu'\lambda) (m'+n',m'+n'+\ell).
\end{align*}
This implies $\Lambda^{\min}(\mu' g\cdot \lambda,\nu'\lambda)=\mt$, and so $\Lambda^{\min}(\mu g\cdot \lambda,\nu\lambda)=\mt$.
\end{proof}

\begin{lem}
\label{L:coma}
Let $n\in \bN$, $z\in \bZ^k$, $p_i$, $q_i\in\bN^k$, $g_i\in G$, $v_i\in\Lambda^0$ ($i=1,\ldots,n$) be such that $p_i-q_i=z$, $(p_i,q_i,g_i)\not\in\Sigma_{G,\Lambda}^{v_i}$.
Then, there exist (a {\rm common})  $a\in \mathbb{N}^k\setminus\{0\}$ and $\lambda_{g_i,v_i} \in v_i\Lambda^a$ such that for any $\mu_i,\nu_i \in \Lambda$ with
$s(\mu_i)=g_i \cdot s(\nu_i)$, $s(\nu_i)=v_i$, and $d(\mu_i)-d(\nu_i)=z$, we have $\Lambda^{\min}(\mu_i(g_i \cdot \lambda_{g_i,v_i}),\nu_i \lambda_{g_i,v_i})=\mt$.
\end{lem}

\begin{proof}
Let $p'=(p_i-q_i)\vee 0$ and $q'=(q_i-p_i)\vee 0$.
From the proof of Lemma \ref{L:min}, for each $(p_i,q_i,g_i)$,
Then there is $x_i\in v_i\Lambda^\infty$ such that
$\sigma^{p'}(x)\ne \sigma^{q'}(g_i\cdot x)$.
So there is $0\ne \ell_i \in \bN^k$ such that $\sigma^{p'}(x_i)(0,\ell_i)\ne \sigma^{q'}(g_i\cdot x_i)(0,\ell_i)$.
Let
$
a:=p'+q'+\sum_{i=1}^n\ell_i
$  and 
$
\lambda_{g_i,v_i}:=x_i(0,a).
$
Then it is easy to see that $\lambda_{g_i,v_i}$ satisfy the required conditions.
\end{proof}

\begin{lem}
\label{L:aK}
Let $p,q\in \bN^k$. If $(p,q,G)\cap\Sigma_{G,\Lambda}=\mt$,
then there exist $a \in \mathbb{N}^k\setminus\{0\}$ and $0<K <1$, such that for any $\mu,\nu \in \Lambda,g \in G$ with $s(\mu)=g \cdot s(\nu)$
and $d(\mu)-d(\nu)=p-q$, we have
\begin{align}
\label{E:aK}
\mathop{\sum_{\lambda \in s(\nu)\Lambda^{na}}}_{\Lambda^{\min}(\mu(g \cdot \lambda),\nu \lambda)\neq\mt}\phi_\Lambda(s_{\nu\lambda} s_{\nu\lambda}^*)
\le K^n \phi_\Lambda(s_\nu s_\nu^*)
\quad (n\in \bN).
\end{align}
\end{lem}

\begin{proof}
Let $(p,q)\in\bN^k$ be such that $(p,q,G)\cap\Sigma_{G,\Lambda}=\mt$.
Since $\varphi$ satisfies the finite-state condition, the set $F:=\{g \vert_\mu: \mu \in \Lambda\}$ is finite.
Notice that $(m,n,h)\not\in\Sigma_{G,\Lambda}$ for any $h\in F$ and $m,n\in\bN^k$ with $m-n=p-q$ by Lemma \ref{L:1toall}.
By Lemma~\ref{L:coma},
there exist $a \in \mathbb{N}^k \setminus \{0\}$ and $\{\lambda_{h,v} \in v\Lambda^{a}\}_{h \in F,v \in \Lambda^0}$ such that for any $\mu,\nu \in \Lambda$ with
$s(\mu)=h \cdot s(\nu)$, $s(\nu)=v$ and $d(\mu)-d(\nu)=p-q$, we have $\Lambda^{\min}(\mu(h \cdot \lambda_{h,v}),\nu \lambda_{h,v})=\mt$. By Theorem~\ref{T:PF}, there exists $0<K <1$ such that $\phi_\Lambda(s_v-s_{\lambda_{h,v}}s_{\lambda_{h,v}}^*) < K \phi_\Lambda(s_v)$ for all $h \in F,v \in \Lambda^0$. Again by Theorem~\ref{T:PF},
\begin{align}
\label{E:ineqK}
x_\Lambda(v)-\rho(\Lambda)^{-a}x_\Lambda(s(\lambda_{h,v}))< K x_\Lambda(v) \qforal h \in F,\ v \in \Lambda^0.
\end{align}

Clearly, the inequality \eqref{E:aK} holds true for $n=0$.
Suppose that its holds true for $n \geq 1$. Then we prove the inequality for $n+1$.
Indeed,
\begin{align*}
&\mathop{\sum_{\lambda \in s(\nu)\Lambda^{(n+1)a}}}_{\Lambda^{\min}(\mu(g \cdot \lambda),\nu \lambda)\neq\mt}\phi_\Lambda(s_{\nu\lambda} s_{\nu\lambda}^*)\\
&=\mathop{\sum_{\omega \in s(\nu)\Lambda^{na}}}_{\Lambda^{\min}(\mu(g \cdot \omega),\nu \omega)\neq\mt}\;
 \mathop{ \sum_{\eta \in s(\nu)\Lambda^{a}}}_{\Lambda^{\min}(\mu(g \cdot \omega)(g \vert_\omega \cdot \eta),\nu \omega \eta)\neq\mt}\phi_\Lambda(s_{\nu\omega\eta} s_{\nu\omega\eta}^*)\\
 &=\mathop{\sum_{\omega \in s(\nu)\Lambda^{na}}}_{\Lambda^{\min}(\mu(g \cdot \omega),\nu \omega)\neq\mt}\;
  \sum_{\eta \ne \lambda_{g|_\omega, s(\omega)}}\phi_\Lambda(s_{\nu\omega\eta} s_{\nu\omega\eta}^*)\\
  &\leq\mathop{ \sum_{\omega \in s(\nu)\Lambda^{na}}}_{\Lambda^{\min}(\mu(g \cdot \omega),\nu \omega)\neq\mt}\Big(\phi_\Lambda(s_{\nu\omega}s_{\nu\omega}^*)-\phi_\Lambda(s_{\nu\omega\lambda_{g\vert_\omega,s(\omega)}}s_{\nu\omega\lambda_{g\vert_\omega,s(\omega)}}^*)\Big)\\
  &=\mathop{\sum_{\omega \in s(\nu)\Lambda^{na}}}_{\Lambda^{\min}(\mu(g \cdot \omega),\nu \omega)\neq\mt}\rho(\Lambda)^{-d(\nu\omega)}\Big( x_\Lambda(s(\omega))
  -\rho(\Lambda)^{-a}x_\Lambda(s(\lambda_{g\vert_\omega,s(\omega)})) \Big)\\
  &<\mathop{\sum_{\omega \in s(\nu)\Lambda^{na}}}_{\Lambda^{\min}(\mu(g \cdot \omega),\nu \omega)\neq\mt}\rho(\Lambda)^{-d(\nu\omega)} Kx_\Lambda(s(\omega))
        \ (\text{by }\eqref{E:ineqK})\\
  &=K\mathop{\sum_{\omega \in s(\nu)\Lambda^{na}}}_{\Lambda^{\min}(\mu(g \cdot \omega),\nu \omega)\neq\mt}\phi_\Lambda(s_{\nu\omega}s_{\nu\omega}^*)\\
  &\le K^{n+1}\phi_\Lambda(s_\nu s_\nu^*) \ (\text{by inductive assumption}).
\end{align*}
Thus \eqref{E:aK} is proved for all $n\in \bN$.
\end{proof}

In Lemma \ref{L:aK}, since $(p,q,G)\cap\Sigma_{G,\Lambda}=\mt$, one necessarily has $p\ne q$. The following lemma essentially handles with the case of $p=q$.

\begin{lem}\label{L:aKp=q}
Let $\nu \in \Lambda$ and $g\ne h \in G$. Then there exist $a \in \mathbb{N}^k\setminus\{0\}$ and $0<K <1$ such that 
\begin{align}
\label{E:aKp=q}
\mathop{\sum_{\lambda \in s(\nu)\Lambda^{na}}}_{g \cdot \lambda =h \cdot \lambda}\phi_\Lambda(s_{\nu\lambda} s_{\nu\lambda}^*) \le K^n \phi_\Lambda(s_\nu s_\nu^*)
\quad (n\in \bN).
\end{align}
\end{lem}

\begin{proof}
For any $v \in \Lambda^0$ and $\lambda \in \Lambda$ with $g \vert_\lambda \neq h \vert_\lambda$, there exists $\omega \in v(\Lambda \setminus \Lambda^0)$ such that $g \vert_\lambda \cdot \omega \neq h \vert_\lambda \cdot \omega$ as $(G,\Lambda)$ is locally faithful. Since $(G,\Lambda)$ satisfies the finite-state condition, the set $F:=\{(g \vert_\lambda,h \vert_\lambda):\lambda \in \Lambda,g \vert_\lambda \neq h \vert_\lambda\}$ is finite. Then there exist $a \in \mathbb{N}^k\setminus\{0\}$ and $\{\lambda_{v,g',h'} \in v\Lambda^a\}_{v \in \Lambda^0,(g',h') \in F}$ such that $g' \cdot \lambda_{v,g',h'} \neq h' \cdot \lambda_{v,g',h'}$ for all $v \in \Lambda^0$ and $(g',h') \in F$. So there is $0<K<1$ such that $\phi_\Lambda(s_v-s_{\lambda_{v,g',h'}}s_{\lambda_{v,g',h'}}^*)<K \phi_\Lambda(s_v)$ for all $v \in \Lambda^0$ and $(g',h') \in F$. By Theorem~\ref{T:PF},
\begin{align}
\label{E:ineqK2}
x_\Lambda(v)-\rho(\Lambda)^{-a}x_\Lambda(s(\lambda_{v,g',h'}))<K x_\Lambda(v)\qforal v \in \Lambda^0, (g',h') \in F.
\end{align}

As before, we show that \eqref{E:aKp=q} holds true for $n+1$ under the inductive assumption that it holds true for $n$.
We compute that
\begin{align*}
&\mathop{\sum_{\lambda \in s(\nu)\Lambda^{(n+1)a}}}_{g \cdot \lambda =h \cdot \lambda}\phi_\Lambda(s_{\nu\lambda} s_{\nu\lambda}^*)\\
&=\mathop{\sum_{\omega \in s(\nu)\Lambda^{na}}}_{g \cdot \omega=h \cdot \omega}
     \mathop{\sum_{\eta \in s(\omega)\Lambda^a}}_{g\vert_\omega \cdot \eta=h \vert_{\omega} \cdot \eta}\phi_\Lambda(s_{\nu\omega\eta}s_{\nu\omega\eta}^*)\\
&=\mathop{\sum_{\omega \in s(\nu)\Lambda^{na}}}_{g \cdot \omega=h \cdot \omega}\Big(\phi_\Lambda(s_{\nu\omega}s_{\nu\omega}^*)
      -\mathop{\sum_{\eta \in s(\omega)\Lambda^a}}_{g\vert_\omega \cdot \eta \neq h \vert_{\omega} \cdot \eta}\phi_\Lambda(s_{\nu\omega\eta}s_{\nu\omega\eta}^*) \Big)\\
&\leq\mathop{\sum_{\omega \in s(\nu)\Lambda^{na}}}_{g \cdot \omega=h \cdot \omega}\phi_\Lambda(s_{\nu\omega}s_{\nu\omega}^*-s_{\nu\omega\lambda_{s(\omega),g\vert_\omega,h\vert_\omega}}s_{\nu\omega\lambda_{s(\omega),g\vert_\omega,h\vert_\omega}}^*)\\
&\quad \text{($g\vert_\omega\neq h\vert_\omega$ since $\varphi$ is pseudo free)}\\
&=\mathop{\sum_{\omega \in s(\nu)\Lambda^{na}}}_{g \cdot \omega=h \cdot \omega}\rho(\Lambda)^{-d(\nu\omega)}(x_{\Lambda}(s(\omega))-\rho(\Lambda)^{-a}x_\Lambda(s(\lambda_{s(\omega),g\vert_\omega,h\vert_\omega})))\\
&\quad \text{(by Theorem~\ref{T:PF})}\\
&\le\mathop{\sum_{\omega \in s(\nu)\Lambda^{na}}}_{g \cdot \omega=h \cdot \omega}\rho(\Lambda)^{-d(\nu\omega)}K x_\Lambda(s(\omega))\ (\text{by }\eqref{E:ineqK2})
\\&=K\mathop{\sum_{\omega \in s(\nu)\Lambda^{na}}}_{g \cdot \omega=h \cdot \omega\}}\phi_\Lambda(s_{\nu\omega}s_{\nu\omega}^*) \text{ (by Theorem~\ref{T:PF})}
\\&\le K^{n+1} \phi_\Lambda(s_\nu s_\nu^*) \ (\text{by inductive assumption}).
\end{align*}
We are done.
\end{proof}

The following lemma can follow from some properties of states combining with Theorem \ref{T:tensor} and Proposition \ref{P:phiLambda}. But we prove it by invoking the Perron-Frobenius theory
directly.

\begin{lem}\label{phi(s_v V_m,n)}
Let $\phi$ be a KMS state on $\O_{G,\Lambda}$ and let $V_{p,q}$ be a periodicity unitary defined in Theorem~\ref{T:Vrep}. Then $\phi(s_v V_{p,q})=x_\Lambda(v)\phi(V_{p,q})$ for all $v \in \Lambda^0$.
\end{lem}

\begin{proof}
For $v \in \Lambda^0$, define $y_v:=\phi(s_v V_{p,q})$. Let $y:=(y_v)_{v \in \Lambda^0}$. Then $y \in \ell^1(\Lambda^0)$. For any $v \in \Lambda^0,1 \leq i \leq k$, we compute that
\begin{align*}
y_v&=\phi(s_v V_{p,q})
\\&=\sum_{\lambda \in v \Lambda^{e_i}}\phi(s_\lambda s_\lambda^* V_{p,q})
\\&=\sum_{\lambda \in v \Lambda^{e_i}}\phi(s_\lambda  V_{p,q} s_\lambda^*) \text{ (as $V_{p,q}$ is central by Theorem \ref{T:Vrep})}
\\&=\sum_{\lambda \in v \Lambda^{e_i}}\rho(\Lambda)^{-e_i}\phi(s_{s(e)}  V_{p,q})
\\&=\rho(\Lambda)^{-e_i} (T_{e_i}y)(v).
\end{align*}
By Theorem~\ref{T:PF}, $y=c x_\Lambda$ for some $c \in \mathbb{C}$. Notice that
\[
c=c\sum_{v \in \Lambda^0}\phi(s_v)=\sum_{v \in \Lambda^0}\phi(s_v V_{p,q})=\phi(V_{p,q}).
\]
So $\phi(s_v V_{p,q})=x_\Lambda(v)\phi(V_{p,q})$ for all $v \in \Lambda^0$.
\end{proof}

\subsection{The KMS simplex of $\O_{G,\Lambda}$}
Let $G$ be an amenable group, and $\Lambda$ be a strongly connected finite $k$-graph. Suppose that $(G,\Lambda)$ is a pseudo free and locally faithful self-similar action which satisfies the finite-state condition. In this section, we will completely describe the structure of the KMS simplex of $(G,\Lambda)$. 

But a general lemma is given first. 

\begin{lem}\label{L:PmuPnu}
Let $\Lambda$ be a self-similar $k$-graph over $G$. Then, for any cycline triple $(\mu,g,\nu)$, we have $s_\mu s_\mu^*=s_\nu s_\nu^*$.
\end{lem}

\begin{proof} Let $(\mu,g,\nu)$ be a cycline triple. Then we compute that
\begin{align*}
&(s_\mu s_\mu^*-s_\nu s_\nu^*)(s_\mu s_\mu^*-s_\nu s_\nu^*)\\
&=s_\mu s_\mu^*+s_\nu s_\nu^*-(s_\mu s_\nu^*s_\nu s_\nu^*+s_\nu s_\nu^* s_\mu s_\mu^*)\\
&=s_\mu s_\mu^*+s_\nu s_\nu^*-\sum_{(\lambda,\omega) \in \Lambda^{\min}(\mu,\nu)}(s_{\mu\lambda}s_{\mu\lambda}^*+s_{\nu\omega}s_{\nu\omega}^*)
\ (\text{by \cite{KP00}})\\
&=s_\mu s_\mu^*+s_\nu s_\nu^*-\sum_{\lambda\in s(\mu)\Lambda^{d(\mu)\lor d(\nu)-d(\mu)}}s_{\mu\lambda}s_{\mu\lambda}^*
     -\sum_{\omega \in s(\nu)\Lambda^{d(\mu)\lor d(\nu)-d(\nu)}}s_{\nu\omega}s_{\nu\omega}^*
     \\
&\quad (\text{as }(\mu, g, \nu)\in \C_{G,\Lambda})
\\&=0 \ (\text{by (CK4)}).
\end{align*}
So $s_\mu s_\mu^*=s_\nu s_\nu^*$.
\end{proof}


\begin{thm}
\label{T:kms}
Let $G$ be an amenable group, $\Lambda$ be a strongly connected finite $k$-graph, and $(G,\Lambda)$ be a pseudo free and locally faithful self-similar action which satisfies the finite-state condition.
If $\phi$ is a KMS state on $\O_{G,\Lambda}$, then for any $\mu,\nu \in \Lambda,g \in G$ with $s(\mu)=g \cdot s(\nu)$, we have
\begin{align*}
\phi(s_\mu u_g s_\nu^*)=\left\{
\begin{array}{ll}
\rho(\Lambda)^{-d(\mu)}x_\Lambda(s_{s(\mu)})\phi(V_{d(\mu),d(\nu)})  & \text{ if $(\mu,g,\nu)\in\C_{G,\Lambda}$},\\
0 & \text{ otherwise}.
\end{array}
\right.
\end{align*}
\end{thm}

\begin{proof}
We divide the proof into two cases.

\underline{Case 1.} $(d(\mu),d(\nu),G)\cap\Sigma_{G,\Lambda}=\mt$.
By Proposition~\ref{P:phiLambda} and Lemma~\ref{L:aK},
there exist $a \in \mathbb{N}^k\setminus\{0\}$ and $0<K <1$, such that for any $n \in \mathbb{N}$, we have
\[
\mathop{\sum_{\lambda \in s(\nu)\Lambda^{na}}}_{\Lambda^{\min}(\mu(g \cdot \lambda),\nu \lambda)\neq\mt}\phi(s_{\nu\lambda} s_{\nu\lambda}^*) \le  K^n \phi(s_\nu s_\nu^*).
\]
So for any $n \in \mathbb{N}$, we calculate that
\begin{align*}
\vert\phi(s_\mu u_g s_\nu^*)\vert
&\leq \sum_{\lambda\in s(\nu)\Lambda^{na}}\vert \phi(s_\mu u_g s_\lambda s_\lambda^* s_\nu^*)\vert\\
&=\sum_{\lambda\in s(\nu)\Lambda^{na}}\vert \phi(s_{\mu(g \cdot \lambda)} u_{g\vert_\lambda} s_{\nu\lambda}^*)\vert\\
&=\mathop{\sum_{\lambda \in s(\nu)\Lambda^{na}}}_{\Lambda^{\min}(\mu(g \cdot \lambda),\nu \lambda)\neq\mt}\vert \phi(s_{\mu(g \cdot \lambda)} 
     u_{g\vert_\lambda} s_{\nu\lambda}^*)\vert
\ (\text{by the KMS condition})\\
&\leq\mathop{\sum_{\lambda \in s(\nu)\Lambda^{na}}}_{\Lambda^{\min}(\mu(g \cdot \lambda),\nu \lambda)\neq\mt}\sqrt{\phi(s_{\nu\lambda}s_{\nu\lambda}^*)\phi(s_{\mu(g \cdot \lambda)}s_{\mu(g \cdot \lambda)}^*)}\\
&\quad \text{(by the Cauchy-Schwarz inequality)}
\\&=\sqrt{\rho(\Lambda)^{d(\nu)-d(\mu)}}\mathop{\sum_{\lambda \in s(\nu)\Lambda^{na}}}_{\Lambda^{\min}(\mu(g \cdot \lambda),\nu \lambda)\neq\mt}\phi(s_{\nu\lambda}s_{\nu\lambda}^*)\\
&\quad \text{(by Theorem~\ref{T:PF} and Proposition \ref{P:phiLambda})}\\
&\le K^n \phi(s_\nu s_\nu^*).
\end{align*}
So $\phi(s_\mu u_g s_\nu^*)=0$.

\underline{Case 2.} $(d(\mu),d(\nu),G)\cap\Sigma_{G,\Lambda}\ne \mt$. That is, there are $h'\in G$ and $v\in\Lambda^0$ such that $(d(\mu),d(\nu),h')\in\Per_{G,\Lambda}^v$.
By Proposition~\ref{P:unicyc}, there exist $h \in G$ and $\lambda \in \Lambda^{d(\nu)}$ such that $(\mu,h,\lambda)$ is a cycline triple. We split into three subcases.

\textit{Subcase 1.} $\nu \neq \lambda$. Then by Lemma~\ref{L:PmuPnu}, we have
\[
\phi(s_\mu u_g s_\nu^*)=\phi(s_\mu s_\mu^* s_\mu u_g s_\nu^*)=\phi(s_\lambda s_\lambda^* s_\mu u_g s_\nu^*)=\phi(s_\mu u_g s_\nu^*s_\lambda s_\lambda^*)=0,
\]
where the 3rd ``=" follows from the KMS condition.

\textit{Subcase 2.} $\nu=\lambda$ and $g \neq h$. By Lemma~\ref{L:aKp=q}, there exist $a \in \mathbb{N}^k\setminus\{0\}$ and $0<K <1$, such that for any $n \in \mathbb{N}$, we have
\[
\mathop{\sum_{\omega \in s(\nu)\Lambda^{na}}}_{g \cdot \omega =h \cdot \omega}\phi(s_{\nu\omega} s_{\nu\omega}^*) \le K^n \phi(s_\nu s_\nu^*).
\]
Then for any $n \in \mathbb{N}$, we calculate that
\begin{align*}
&\vert\phi(s_\mu u_g s_\nu^*)\vert\\
&\leq\sum_{\omega \in s(\nu)\Lambda^{na}}\vert \phi(s_{\mu (g \cdot \omega)} u_{g\vert_\omega} s_{\nu\omega}^*) \vert\\
&=\mathop{\sum_{\omega \in s(\nu)\Lambda^{na}}}_{g \cdot \omega=h \cdot \omega}\vert \phi(s_{\mu (g \cdot \omega)} u_{g\vert_\omega} s_{\nu\omega}^*) \vert+
 \mathop{\sum_{\omega \in s(\nu)\Lambda^{na}}}_{g \cdot \omega \neq h \cdot \omega}\vert \phi(s_{\mu (g \cdot \omega)} u_{g\vert_\omega} s_{\nu\omega}^*) \vert\\
&=\mathop{\sum_{\omega \in s(\nu)\Lambda^{na}}}_{g \cdot \omega=h \cdot \omega}\vert \phi(s_{\mu (g \cdot \omega)} u_{g\vert_\omega} s_{\nu\omega}^*) \vert
\ \text{(by an argument similar to Subcase 1)}
\\&\leq \mathop{\sum_{\omega \in s(\nu)\Lambda^{na}}}_{g \cdot \omega=h \cdot \omega}\sqrt{\phi(s_{\nu\omega}s_{\nu\omega}^*)
     \phi(s_{\mu(g \cdot \omega)}s_{\mu(g \cdot \omega)}^*)}
\ \text{(by the Cauchy-Schwarz inequality)}
\\&=\mathop{\sum_{\omega \in s(\nu)\Lambda^{na}}}_{g \cdot \omega=h \cdot \omega}\sqrt{\rho(\Lambda)^{d(\nu)-d(\mu)}}\phi(s_{\nu\omega}s_{\nu\omega}^*)
\ \text{(by Theorem~\ref{T:PF} and Proposition~\ref{P:phiLambda})}
\\&\le \sqrt{\rho(\Lambda)^{d(\nu)-d(\mu)}} K^n \phi(s_\nu s_\nu^*).
\end{align*}
So $\phi(s_\mu u_g s_\nu^*)=0$.

\textit{Subcase 3.} $\nu=\lambda$ and $g=h$. Then $(\mu,g,\nu)$ is a cycline triple.
Thus we have
\begin{align*}
\phi(s_\mu u_g s_\nu^*)&=\phi(s_\mu s_\mu^* V_{d(\mu),d(\nu)})\ (\text{by }\eqref{E:Vmn})\\
&=\rho(\Lambda)^{-d(\mu)}\phi( s_\mu^* V_{d(\mu),d(\nu)} s_\mu)\ (\text{by the KMS condition})\\
&=\rho(\Lambda)^{-d(\mu)}\phi( s_\mu^* s_\mu V_{d(\mu),d(\nu)})\ (\text{by Theorem~\ref{T:Vrep}})\\
&=\rho(\Lambda)^{-d(\mu)}\phi(s_{s(\mu)}V_{d(\mu),d(\nu)})\\
&=\rho(\Lambda)^{-d(\mu)} x_\Lambda(s(\mu))\phi(V_{d(\mu),d(\nu)}) \ (\text{by Lemma \ref{phi(s_v V_m,n)}}).
\end{align*}
This ends the proof.
\end{proof}


\begin{thm}\label{T:Mstate}
Keep the same conditions as in Theorem \ref{T:kms}. Then the KMS simplex of $\mathcal{O}_{G,\Lambda}$ is affinely isomorphic to the convex compact subset $\S$ of the tracial state space of $\mathcal{M}_{G,\Lambda}$
 satisfying
\begin{align*}
\phi(s_\mu u_g s_\nu^*)=\rho(\Lambda)^{-d(\mu)}x_\Lambda(s(\mu))\phi(V_{d(\mu),d(\nu)})\qforal (\mu,g,\nu) \in \C_{G,\Lambda}.
\end{align*}
\end{thm}

\begin{proof}
It is clear that $\S$ is a convex compact subset of the tracial state space of $\mathcal{M}_{G,\Lambda}$.

By Theorem~\ref{T:kms} one can define a mapping $\Phi$ from KMS simplex of $\O_{G,\Lambda}$ to $\S$ by $\Phi(\phi):=\phi \vert_{\mathcal{M}_{G,\Lambda}}$.
Also, Theorem~\ref{T:kms} implies that $\Phi$ is injective. It is enough to show that $\Phi$ is surjective. Fix $\phi \in \S$.
By Theorem \ref{T:Exp}, there exists a faithful expectation $E:\mathcal{O}_{G,\Lambda} \to \mathcal{M}_{G,\Lambda}$ such that $E(s_\mu u_g s_\nu^*)=0$
if $(\mu,g,\nu)\not\in\C_{G,\Lambda}$. We show that $\phi \circ E$ is a KMS state of $\mathcal{O}_{G,\Lambda}$. Fix $\mu,\nu,\alpha,\beta \in \Lambda, g,h \in G$ with $s(\mu)=g \cdot s(\nu)$ and $s(\alpha)=h \cdot s(\beta)$. We must show that
\begin{align}
\label{E:phiE}
(\phi \circ E)(s_\mu u_g s_\nu^* s_\alpha u_h s_\beta^*)=\rho(\Lambda)^{d(\nu)-d(\mu)}(\phi \circ E)(s_\alpha u_h s_\beta^* s_\mu u_g s_\nu^* ).
\end{align}

 If $d(\mu)-d(\nu)+d(\alpha)-d(\beta) \notin \Per_{G,\Lambda}$, then by Theorem~\ref{T:Pergroup}, we deduce that
\[
(\phi \circ E)(s_\mu u_g s_\nu^* s_\alpha u_h s_\beta^*)=(\phi \circ E)(s_\alpha u_h s_\beta^* s_\mu u_g s_\nu^* )=0.
\]

If $d(\mu)-d(\nu)+d(\alpha)-d(\beta) \in \Per_{G,\Lambda}$, we may also assume that $d(\alpha) \geq d(\nu)$ and $d(\beta)\ge d(\mu)$ by the Cuntz-Krieger relations. 
Write $\alpha=\nu' \lambda$ and $\beta=\mu' \omega$ such that $d(\mu)=d(\mu')$ and $d(\nu)=d(\nu')$. We split into four cases.

\underline{Case 1.} $\nu=\nu'$ and $\mu=\mu'$. Then
\[
s_\mu u_g s_\nu^* s_\alpha u_h s_\beta^*=s_{\mu(g \cdot \lambda)}u_{g\vert_\lambda h}s_{\mu\omega}^*, s_\alpha u_h s_\beta^* s_\mu u_g s_\nu^*=s_{\nu \lambda}u_{h (g\vert_{g^{-1} \cdot \omega})}s_{\nu(g^{-1} \cdot \omega)}^*.
\]
Observe that 
$(\mu(g \cdot \lambda),g\vert_\lambda h,\mu\omega)$ is a cycline triple if and only if so is $(\nu \lambda,h$ $g\vert_{g^{-1} \cdot \omega},\nu(g^{-1} \cdot \omega))$. 
If $(\mu(g \cdot \lambda),g\vert_\lambda h,\mu\omega)$ is not a cycline triple, then
\[
(\phi \circ E)(s_\mu u_g s_\nu^* s_\alpha u_h s_\beta^*)=(\phi \circ E)(s_\alpha u_h s_\beta^* s_\mu u_g s_\nu^* )=0.
\]
If $(\mu(g \cdot \lambda),g\vert_\lambda h,\mu\omega)$ is cycline, then
\begin{align*}
&(\phi \circ E)(s_\mu u_g s_\nu^* s_\alpha u_h s_\beta^*)
  =\phi(s_{\mu(g \cdot \lambda)}u_{g\vert_\lambda h}s_{\mu\omega}^*)
\\&=\rho(\Lambda)^{-d(\mu)-d(\lambda)}\phi(s_{s(g \cdot \lambda)}V_{d(\lambda),d(\omega)})
\ \text{(By Theorem~\ref{T:Vrep})}
\\&=\rho(\Lambda)^{-(d(\mu)-d(\nu))}\rho(\Lambda)^{-d(\nu)-d(\lambda)}\phi(s_{s(\lambda)}V_{d(\lambda),d(\omega)})
\\&=\rho(\Lambda)^{-(d(\mu)-d(\nu))}\phi(s_{\nu \lambda}u_{h (g\vert_{g^{-1} \cdot \omega})}s_{\nu(g^{-1} \cdot \omega)}^*)
\ \text{(By Theorem~\ref{T:Vrep})}
\\&=\rho(\Lambda)^{-(d(\mu)-d(\nu))}(\phi \circ E)(s_\alpha u_h s_\beta^* s_\mu u_g s_\nu^* ).
\end{align*}

\underline{Case 2.} $\nu=\nu'$ and $\mu \neq \mu'$. Then
\begin{align*}
(\phi \circ E)(s_\mu u_g s_\nu^* s_\alpha u_h s_\beta^*)=(\phi\circ E)(s_{\mu(g \cdot \lambda)}u_{g\vert_\lambda h}s_{\mu'\omega}^*)=0
\end{align*}
because $(\mu(g \cdot \lambda),g\vert_\lambda h,\mu'\omega)$ is not a cycline triple. On the other hand, $(\phi \circ E)(s_\alpha u_h s_\beta^* s_\mu u_g s_\nu^* )=0$.

\underline{Case 3.} $\nu \neq \nu'$ and $\mu=\mu'$. Similar to Case 2, we have
\[
(\phi \circ E)(s_\mu u_g s_\nu^* s_\alpha u_h s_\beta^*)=(\phi \circ E)(s_\alpha u_h s_\beta^* s_\mu u_g s_\nu^* )=0.
\]

\underline{Case 4.} $\nu \neq \nu'$ and $\mu \neq \mu'$. Then
\[
(\phi \circ E)(s_\mu u_g s_\nu^* s_\alpha u_h s_\beta^*)=(\phi \circ E)(s_\alpha u_h s_\beta^* s_\mu u_g s_\nu^* )=0.
\]

Therefore, $\phi \circ E$ satisfies \eqref{E:phiE}, and so it is a KMS state of $\mathcal{O}_{G,\Lambda}$. Clearly $(\phi \circ E)|_{\M_{G,\Lambda}}=\phi$.
Hence $\Phi$ is an affine isomorphism.
\end{proof}

\begin{thm}
\label{T:KMSPer}
Keep the same conditions as in Theorem \ref{T:kms}.
Then the KMS simplex of $\mathcal{O}_{G,\Lambda}$ is affinely isomorphic to the tracial state space of $\ca(\Per_{G,\Lambda})$.
\end{thm}

\begin{proof}
Let $\fA:=\ca(\{V_{m,n}: m, n\in\bN^k, m-n\in\Per_{G,\Lambda}\})$.
By Theorems \ref{T:embed} and \ref{T:Mstate}, it suffices to show that the subset $\S$ of the tracial state space of $\mathcal{M}_{G,\Lambda}$ in Theorem~\ref{T:Mstate} is affinely isomorphic to the tracial state space of $\fA$.

Define $\Psi:\S \to \{\text{Tracial states of $\fA$}\}$ by $\Psi(\phi):=\phi \vert_\fA$.
By the definition of $\S$, $\Psi$ is injective. We only need to show that $\Psi$ is surjective. Fix a tracial state $\phi$ of $\fA$.
Recall from Theorem~\ref{T:PF} that $\phi_\Lambda$ is the Perron-Frobenius state on $\D_\Lambda$. So  $\phi \otimes \phi_\Lambda$
is the unique tracial state on
$\fA \otimes \D_\Lambda$ such that $(\phi \otimes \phi_\Lambda)(a \otimes b)=\phi(a)\phi_\Lambda(b)$ for all
$a \in \fA, b \in \D_\Lambda$. By Theorems~\ref{T:embed} and \ref{T:tensor}, $\phi \otimes \phi_\Lambda$ can be regarded as a tracial state on $\mathcal{M}_{G,\Lambda}$. It is straightforward to see that $\phi \otimes \phi_\Lambda \in \S$ and $(\phi \otimes \phi_\Lambda) \vert_{\fA}=\phi$. So $\Psi$ is injective. Hence $\Psi$ is an affine isomorphism.
\end{proof}

\subsection{Two applications}

Two applications will given in this short subsection.

The first one generalizes \cite[Remark 7.2]{HLRS15}. 

\begin{cor}
\label{C:Per}
$\Per_{G,\Lambda}$ is a subgroup of $\{n\in \bZ^k:\rho(\Lambda)^n=1\}$.
\end{cor}

\begin{proof}
It is obvious that $\{n\in \bZ^k:\rho(\Lambda)^n=1\}$ is a subgroup of $\bZ^k$.
By Theorem \ref{T:Pergroup}, $\Per_{G,\Lambda}$ is a group. 
In what follows, we show that 
for $(\mu,g,\nu)\in\C_{G,\Lambda}$, one has $\rho(\Lambda)^{d(\mu)}=\rho(\Lambda)^{d(\nu)}$.
By Theorems \ref{T:embed} and \ref{T:KMSPer}, there exists a KMS state $\phi$ on $\mathcal{O}_{G,\Lambda}$ such that $\phi(V_{d(\mu),d(\nu)}) \neq 0$. 
So
\begin{align*}
&\rho(\Lambda)^{-d(\mu)}x_\Lambda(s(\mu))\phi(V_{d(\mu),d(\nu)})\\
&=\phi(s_\mu u_g s_\nu^*)\ \text{(By Theorem~\ref{T:kms})}
\\&=\overline{\phi(s_\nu u_{g^{-1}} s_\mu^*)} \ \text{(by Lemma \ref{L:equi})}
\\&=\rho(\Lambda)^{-d(\nu)}x_\Lambda(s(\nu))\overline{\phi(V_{d(\nu),d(\mu)})}\ \text{(By Theorem~\ref{T:kms})}
\\&=\rho(\Lambda)^{-d(\nu)}x_\Lambda(s(\mu))\overline{\phi(V_{d(\mu),d(\nu)}^*)}\ \text{(By Theorem~\ref{T:Vrep})}
\\&=\rho(\Lambda)^{-d(\nu)}x_\Lambda(s(\mu))\phi(V_{d(\mu),d(\nu)}).
\end{align*}
Hence $\rho(\Lambda)^{d(\mu)}=\rho(\Lambda)^{d(\nu)}$, and we are done.
\end{proof}

In general, $\Per_{G,\Lambda}$ is a proper subgroup of $\{n\in \bZ^k:\rho(\Lambda)^n=1\}$ even for $G$-periodic $\Lambda$ (see, e.g., \cite{DY091}).

The second application contributes one more item on the list of characterizing the $G$-aperiodicity of $(G,\Lambda)$, which
generalizes the corresponding parts \cite[Theorem 11.1]{HLRS15}.

\begin{cor}
\label{C:uniKMS}
The following are equivalent:
\begin{enumerate}
\item
$(G,\Lambda)$ is $G$-aperiodic;
\item
$\mathcal{O}_{G,\Lambda}$ is simple;
\item
$\Per_{G,\Lambda}=\{0\}$;
\item
there exists a unique KMS state on $\mathcal{O}_{G,\Lambda}$.
\end{enumerate}
\end{cor}

\begin{proof}
(i)$\iff$(ii) follows from \cite[Theorem~6.6]{LY17_2};
(i)$\iff$(iii) is from Theorem \ref{T:chaGape}; and
(i)$\iff$(iv) is from Theorems~\ref{T:chaGape1} and \ref{T:KMSPer}.
\end{proof}


\section{Examples}

\label{S:eg}

In this section, we apply our results to compute the KMS states of several classes of self-similar $k$-graph C*-algebras $\O_{G,\Lambda}$
under our standing assumption \eqref{E:KMSass}.

Unless otherwise stated, as in Section \ref{S:KMS}, let $G$ be an amenable group, $\Lambda$ a strongly connected finite $k$-graph, and $(G,\Lambda)$ a pseudo free locally faithful self-similar action which satisfies the finite-state condition.

\subsection{Strongly connected finite $k$-graph C*-algebras} Let $G$ be the trivial group. Note that the assumption \eqref{E:KMSass} is redundant in this case. Then 
Theorem \ref{T:KMSPer} implies the main result of \cite{HLRS15}
for strongly connected finite $k$-graph C*-algebras.

\subsection{Rational independence}

\begin{thm}\label{KMS rat ind}
Suppose that $\{\ln (\rho(T_{e_i}))\}_{i=1}^{k}$ is rationally independent. Then there exists a unique KMS state on $\mathcal{O}_{G,\Lambda}$.
\end{thm}

\begin{proof}
Since $\{\ln (\rho(T_{e_i}))\}_{i=1}^{k}$ is rationally independent, we have that  $\{n\in \bN^k:p,q \in \mathbb{N}^k,\rho(\Lambda)^n=1\}=\{0\}$. By Corollary~\ref{C:Per}, $\Per_{G,\Lambda}=\{0\}$. Then, by Theorem~\ref{T:chaGape1}, $(G,\Lambda)$ is $G$-aperiodic. By Theorem~\ref{C:uniKMS}, there exists a unique KMS state on $\mathcal{O}_{G,\Lambda}$.
\end{proof}

\subsection{The degenerate property}

\begin{defn}
\label{D:deg}
A self-similar action $(G, \Lambda)$ is said to have
the \emph{degenerate property} if for any $g \in G$, any $v \in \Lambda^0$, there exists $\mu \in v\Lambda$ such that $g \vert_\mu=1_G$.
\end{defn}


\begin{thm}\label{KMS_1 state degenerate property}
Suppose that $(G,\Lambda)$ satisfies the degenerate property. Then
\begin{enumerate}
\item $\Per_{G,\Lambda}=\Per_\Lambda$;
\item
the KMS simplex of $\mathcal{O}_{G,\Lambda}$ is affinely isomorphic to the tracial state space of $\ca(\Per_\Lambda)$.
\end{enumerate}
\end{thm}

\begin{proof}
(i) Obviously, $\Per_\Lambda\subseteq\Per_{G,\Lambda}$. We show $\Per_{G,\Lambda}\subseteq\Per_\Lambda$. Let $(\mu,g,\nu)\in\C_{G,\Lambda}$. Since $(G,\Lambda)$  satisfies the degenerate property, there exists $\lambda \in s(\nu)\Lambda$ such that $g \vert_\lambda=1_G$. Notice that $g \cdot (\lambda s(\lambda))=(g \cdot \lambda)s(\lambda)$. So $s(g \cdot \lambda)=s(\lambda)$. Then for any $x \in s(\lambda)\Lambda^\infty$, we have
\[
\mu(g \cdot \lambda)x=\mu( g \cdot (\lambda x))=\nu \lambda x.
\]
Thus $(\mu(g \cdot \lambda),\nu\lambda)\in\C_\Lambda$. Therefore $\Per_{G,\Lambda} \subseteq\Per_\Lambda$.

(ii) This follows immediately from (i) above and Theorem~\ref{T:KMSPer}.
\end{proof}


\subsection{Products of odometers C*-algebras}
\label{SS:PO}

In this subsection, let $\Lambda$ be a single-vertex $k$-graph. A product of odometers is a self-similar action $(\bZ, \Lambda)$ defined below:
\begin{enumerate}
\item $\Lambda^{e_i}:=\{x_{\fs}^i\}_{\fs =0}^{n_i-1},1\le i\le k,n_i>1$;
\item $1\cdot {x_\fs^i}=x_{(\fs+1)\ \text{mod } n_i}^i,1\le i\le k,0\le \fs\le n_i-1$;
\item $1|_{x_\fs^i}= \begin{cases}
    0 &\text{ if }0\le \fs< n_i-1\\
    1 &\text{ if }\fs=n_i-1
\end{cases}
\ (1\le i\le k);$
\item $x^i_\fs x^j_\ft = x^j_{\ft'} x^i_{\fs'} \text{ if } 1\leq i<j \leq k, \fs+\ft n_i=\ft'+\fs' n_j$.
\end{enumerate}
Then one can check that $(\bZ,\Lambda)$ is pseudo free, locally faithful, satisfies the finite-state condition and the degenerate property.

\begin{rem}
When $k=1$, the example above is the classical odometer and the KMS states of $\O_{\mathbb{Z},\Lambda}$ was studied by Laca-Raeburn-Ramagge-Whittaker in \cite{LRRW14}. When $k=2$, this example was constructed by Brownlowe-Ramagge-Robertson-Whittaker in terms of Zappa-Sz\'ep products of semigroups in \cite{BRRW14}. For general $k$, we studied the properties of $\O_{\mathbb{Z},\Lambda}$ intensively in \cite{LY17}.
\end{rem}

The theorem below gives the periodicity group $\Per_{\bZ,\Lambda}$ explicitly, which provides an extremal case in the periodicity theory (cf.~Corollary \ref{C:Per}).
The approach used in the proof below is indirected. A direct approach is given in Appendix~A.


\begin{thm}
\label{T:odoPer}
Let $n:=(n_1,\ldots, n_k)$. Then 
\[
\Per_{\bZ,\Lambda}=\Per_\Lambda=\{p\in \bZ^k:n^p=1\}.
\]
\end{thm}

\begin{proof}
Notice that $\rho(\Lambda)=n$.
By Corollary~\ref{C:Per} and Theorem~\ref{KMS_1 state degenerate property}, it remains to show $\{p\in \bZ^k:n^p=1\}\subseteq \Per_\Lambda$.
If $\Per_\Lambda=\{0\}$, then by Theorem~\ref{C:uniKMS}, $\mathcal{O}_{G,\Lambda}$ is simple. By \cite[Theorem~5.12]{LY17},
$\{\ln n_i\}_{i=1}^k$ is rationally independent, which gives $\{p\in \bZ^k:n^p=1\}=\{0\}$ too.
So we may assume that $\Per_\Lambda \neq \{0\}$.
To the contrary, assume that
$\Per_\Lambda \subsetneq \{p\in \bZ^k:n^p=1\}$. Let $\{n^i=p^i-q^i: p^i, q^i\in\bN^k\}_{i=1}^{s}$ be a basis of $\Per_\Lambda$,
and $p^0-q^0 \in \{p-q:p,q \in \mathbb{N}^k,\rho(\Lambda)^p=\rho(\Lambda)^q\} \setminus \Per_\Lambda$. We may assume that $p^i\wedge q^i=0$ for all $0 \leq i \leq s$
by properties of $\Per_\Lambda$.

Inspired by \cite{Cun08}, there exists a representation $\pi:\mathcal{O}_\Lambda \to B(\ell^2(\mathbb{Z}))$ such that $\pi(x_{\fs}^i)(\delta_m)=\delta_{\fs+n_i m}$ for all $1 \leq i \leq k,0 \leq \fs \leq n_i-1,m \in \mathbb{Z}$.
One can verify that, for any $0\ne p,q \in \mathbb{N}^k$ with $p\wedge q=0$ and $n^p=n^q$, there exists a unique bijection $\rho:\Lambda^p \to \Lambda^q$ such that $\pi(s_{\mu}-s_{\rho(\mu)})=0$ for all $\mu \in \Lambda^p$. Furthermore, if $p-q \in \Per_\Lambda$, let $\theta_{p,q}:\Lambda^p \to \Lambda^q$ be the bijection given in Remark~\ref{R:thetapq} (in the case that $G$ is trivial). Then, by \cite[Theorem~7.1]{DY09}, for any $\mu \in \Lambda^p$ and $\nu \in \Lambda^q$, we have $\mu\nu=\theta_{p,q}(\mu)\theta_{p,q}^{-1}(\nu)$. Then $\pi(s_{\rho(\mu)} s_{\rho^{-1}(\nu)})=\pi(s_\mu s_\nu)=\pi(s_{\theta_{p,q}(\mu)}s_{\theta_{p,q}^{-1}(\nu)})$. So $\pi(s_{\theta_{p,q}(\mu)}^*s_{\rho(\mu)} s_{\rho^{-1}(\nu)})=\pi(s_{\theta_{p,q}^{-1}(\nu)})$. The fact that $\pi(s_{\theta_{p,q}^{-1}(\nu)})\neq 0$ forces $\rho(\mu)=\theta_{p,q}(\mu)$
as $d(\rho(\mu))=d(\theta_{p,q}(\mu))(=q)$. Hence $\rho=\theta_{p,q}$.

For $0 \leq i \leq s$, denote by $\rho_i:\Lambda^{p^i} \to \Lambda^{q^i}$ the map we just described and denote by a unitary $W_i:=\sum_{\mu \in \Lambda^{p^i}}s_\mu s_{\rho_i(\mu)}^*$. Since $W_i-1_{\mathcal{O}_\Lambda}=\sum_{\mu \in \Lambda^{p^i}}(s_\mu-s_{\rho_i(\mu)})s_{\rho_i(\mu)}^*, W_i -1_{\mathcal{O}_\Lambda}\in \ker(\pi)$. Denote by $\I$ the closed two-sided ideal of $\mathcal{O}_\Lambda$ generated by $\{W_i-1_{\mathcal{O}_\Lambda}\}_{i=1}^{s}$. Then $\I\subseteq \ker\pi$.
But by \cite[Theorem~8.3]{DY09}, $\O_\Lambda/\I$ is simple. So $\ker\pi=\I$.
In particular, $W_0-1_{\mathcal{O}_\Lambda} \in \I$. By \cite[Corollary~8.7]{DY09}, $\{W_i\}_{i=1}^{s}$ is in the center of $\mathcal{O}_{\Lambda}$. So $\mathcal{O}_\Lambda\{(W_i-1_{\mathcal{O}_\Lambda})\}_{i=1}^{s}$ can approximates $W_0-1_{\mathcal{O}_\Lambda}$. Denote by $E$ the expectation obtained from Proposition~\ref{P:CycCartan}.
So $E(\mathcal{O}_\Lambda\{(W_i-1_{\mathcal{O}_\Lambda})\}_{i=1}^{s})$ (and hence $\mathcal{O}_\Lambda\{(W_i-1_{\mathcal{O}_\Lambda})\}_{i=1}^{s}$) can approximates $E(W_0-1_{\mathcal{O}_\Lambda})=-1_{\mathcal{O}_\Lambda}$ because $p^0-q^0 \notin \Per_\Lambda$. Hence $1_{\mathcal{O}_\Lambda} \in \I$ which is a contradiction. Therefore $\Per_\Lambda=\{p\in \bZ^k:n^p=1\}$.
\end{proof}

By Theorems \ref{T:KMSPer} and \ref{T:odoPer} we have

\begin{cor}
The KMS simplex of $\mathcal{O}_{\bZ,\Lambda}$ is affinely isomorphic to the
tracial state space of $\ca(\{p\in \mathbb{Z}^k:n^p=1\})$.
\end{cor}

\subsection{Unital Katsura algebras}
Let $\Lambda$ be a $1$-graph such that $T_{e_1}(v,v) \geq 2$ for all $v \in \Lambda^0$.  For  $v,w \in \Lambda^0$, if $T_{e_1}(v,w) \neq 0$,   
we list the edges from $w$ to $v$ as $v\Lambda^{e_1}w:=\{(v,w,n):0 \leq n <T_{e_1}(v,w)\}$.

Let $B \in M_{\Lambda^0}(\mathbb{Z})$ satisfy the following properties: 
\begin{itemize}
\item 
$B(v,v)=1$ for all $v \in \Lambda^0$;
\item $\vert B(v,w)\vert\leq T_{e_1}(v,w)$ for all $v, w \in \Lambda^0$;
\item 
$B(v,w)=0\iff T_{e_1}(v,w)=0$ for all $v,w \in \Lambda^0$.
\end{itemize}
We construct a self-similar action $(\bZ, \Lambda)$ using $B$ as follows: 
\begin{itemize}
\item $g\cdot v=v$ for all $g\in \bZ$ and $v\in \Lambda^0$. 
\item For $g \in \mathbb{Z}$, $v,w \in \Lambda^0$ with $T_{e_1}(v,w) \neq 0$, and $0 \leq m <T_{e_1}(v,w)$, let $h \in \mathbb{Z}$ and $0 \leq n <T_{e_1}(v,w)$ be the unique integers such that $g B(v,w)+m=h T_{e_1}(v,w)+n$. Define 
\[
g \cdot (v,w,m):=(v,w,n) \text{ and }g \vert_{(v,w,m)}:=h.
\]
\end{itemize}
Then one can check that $(\bZ,\Lambda)$ is pseudo free, locally faithful, and satisfies the finite-state condition and the degenerate property.


\begin{rem}
Katsura in \cite{Kat08_1} constructed a unital C*-algebra $\mathcal{O}_{T_{e_1},B}$ which is the universal C*-algebra generated by a family of mutually orthogonal projections $\{q_v\}_{v \in \Lambda^0}$, a family of partial unitaries $\{r_v\}_{v \in \Lambda^0}$ with $r_v^*r_v=r_vr_v^*=q_v$, and a family of partial isometries $\{t_{v,w,z}:T_{e_1}(v,w) \neq 0,z \in \mathbb{Z}\}$ satisfying the following properties: 
\begin{enumerate}
\item $t_{v,w,z}r_w=t_{v,w,z+T_{e_1}(v,w)}$ and $r_v t_{v,w,z}=t_{v,w,z+B(v,w)}$ for all $v,w \in \Lambda^0$ and $z\in \mathbb{Z}$ with $T_{e_1}(v,w) \neq 0$;
\item $t_{v,w,z}^*t_{v,w,z}=q_w$ for all $z\in \bZ$ and $v,w \in \Lambda^0$ with $T_{e_1}(v,w) \neq 0$;
\item $q_v=\sum\limits_{\{w \in \Lambda^0:T_{e_1}(v,w) \neq 0\}}\sum\limits_{n=1}^{T_{e_1}(v,w)}t_{v,w,n}t_{v,w,n}^*$ for all $v \in \Lambda^0$.
\end{enumerate}
Exel-Pardo in \cite{EP17} showed that $\mathcal{O}_{\mathbb{Z},\Lambda}$ is isomorphic to $\mathcal{O}_{T_{e_1},B}$ via the identifications 
$s_v \mapsto q_v$,  $s_{v,w,n} \mapsto t_{v,w,n+1}$, and $u_z \mapsto (\sum_{v \in \Lambda^0}r_v)^z$ for all 
$z\in \bZ$, $v,w \in \Lambda^0$ with $T_{e_1}(v,w) \neq 0$, and $0 \leq n <T_{e_1}(v,w)$. 
\end{rem}

\begin{rem}
Katsura in \cite{Kat08_1}, and Exel-Pardo in \cite{EP17} both showed that $\mathcal{O}_{\mathbb{Z},\Lambda} (\cong \mathcal{O}_{T_{e_1},B})$ is simple, separable, amenable, purely infinite, and satisfies the UCT. Here we would like to comment that this can also be obtained by invoking our results in \cite{LY17_2}. In fact, By \cite[Theorem~6.6]{LY17_2}, $\mathcal{O}_{\mathbb{Z},\Lambda}$ is amenable and satisfies the UCT. Since $T_{e_1}(v,v) \geq 2$ for all $v \in \Lambda^0$, $\Lambda$ is aperiodic. Since $(\bZ,\Lambda)$ satisfies the degenerate property, then Theorem~\ref{KMS_1 state degenerate property} implies that $(\bZ, \Lambda)$ is $G$-aperiodic. Applying \cite[Theorem~6.6]{LY17_2} again yields that
 $\mathcal{O}_{T_{e_1},B}$ is simple. Since $T_{e_1}(v,v) \geq 2$ for all $v \in \Lambda^0$, there exist no nonzero graph traces on $\Lambda$. Thus 
 $\mathcal{O}_{\mathbb{Z},\Lambda}$ is purely infinite by \cite[Theorem~6.13]{LY17_2}.
\end{rem}

\begin{thm}
There exists a unique KMS state on $\mathcal{O}_{\mathbb{Z},\Lambda}$.
\end{thm}
\begin{proof}
We can invoke either Katsura's result in \cite{Kat08_1}, or Exel-Pardo's result in \cite{EP17}, or our result in \cite{LY17_2} to deduce that $\mathcal{O}_{\mathbb{Z},\Lambda}$ is simple. By Corollary~\ref{C:uniKMS}, there exists a unique KMS state on $\mathcal{O}_{\mathbb{Z},\Lambda}$.
\end{proof}


\appendix
\section{Another approach to compute $\Per_{\bZ,\Lambda}$ for products of odometers}

Keep the same notation in Subsection \ref{SS:PO}. The main purpose of this appendix is to compute $\Per_{\Lambda}$ using a direct and elementary approach.

Put $[n_i]:=\{0,1,\ldots, n_i-1\}$.
For $w=s_1\cdots s_\ell\in\Lambda^{\ell e_i}$ with $s_1,\ldots, s_\ell\in[n_i]$, let
\[
[w]_{n_i}:=s_1+s_2 n_i+\cdots + s_\ell n_i^{\ell-1},
\]
which can be regarded as the representation of $w$ in base-$n_i$ number system.
For $w_i\in\Lambda_i$, the $1$-graph of colour $i$ with $1\le i\le \fk$, let
\[
\mathbf{n}^{|(w_1,\ldots,w_\fk)|}:=\Pi_{i=1}^\fk n_i^{|w_i|}.
\]

\begin{lem}
\label{L:odometer}
Let $p,q\in\bN$, $\{i_\ell\}_{\ell=1}^p$ and $\{j_\fk\}_{\fk=1}^q$ be two disjoint non-empty sets in $\{1,\ldots,k\}$.
Let $u_\ell\in\Lambda_{{i_\ell}}$ for $1\le \ell\le p$, and $v_\fk\in\Lambda_{{j_\fk}}$ for $1\le \fk\le q$. Then
\begin{align*}
x_{u_1}^{i_1}\cdots x_{u_p}^{i_p}\, x_{v_1}^{j_1}\cdots x_{v_q}^{j_q}
=x_{v_1'}^{j_1}\cdots x_{v_q'}^{j_q}\, x_{u_1'}^{i_1}\cdots x_{u_p'}^{i_p},
\end{align*}
where $u_\ell'\in\Lambda_{i_\ell}$ with $d(u_\ell)=d(u_\ell')$ for $1\le \ell\le p$
and $v_\fk'\in \Lambda_{j_\fk}$ with $d(v_\fk)=d(v_\fk')$ for $1\le \fk\le q$ are uniquely determined by the following identity:
\begin{align}
\nonumber
&\sum_{\ell=1}^p[u_\ell]_{n_{i_\ell}}\mathbf{n}^{|(u_1,\ldots, u_{\ell-1})|}
+\left(\sum_{\fk=1}^q[v_\fk]_{n_{j_\fk}}\mathbf{n}^{|(v_1,\ldots, v_{\fk-1})|}\right)\mathbf{n}^{|(u_1,\ldots, u_p)|}\\
\label{E:odometer}
=&\sum_{\fk=1}^q[v_\fk']_{n_{j_\fk}}\mathbf{n}^{|(v_1,\ldots, v_{\fk-1})|}
  +\left(\sum_{\ell=1}^p[u_\ell']_{n_{i_\ell}}\mathbf{n}^{|(u_1,\ldots, u_{\ell-1})|}  \right)\mathbf{n}^{|(v_1,\ldots, v_q)|}.
\end{align}
Here $\mathbf{n}^{|u_0|}:=1$ and $\mathbf{n}^{|v_0|}:=1$.
\end{lem}

\begin{proof}
We first show that the lemma holds true for $p=q=1$. If $|u_1|=|v_1|=1$, clearly \eqref{E:odometer} is satisfied by our assumption.
Let assume that \eqref{E:odometer} holds true for $|u_1|=1$ and $|v_1|\le L$.
Let $1\le j\ne i_1\le k$.
Then
\begin{align*}
x^{i_1}_{u_1} x^{j_1}_{v_1} x_\ft^{j_2}= x^{j_1}_{v_1'} x^{i_1}_{u_1''} x_\ft^{j_2}=x^{j_1}_{v_1'} x_{\ft'}^{j_2}x^{i_1}_{u_1'},
\end{align*}
where
\begin{align*}
u_1+[v_1]_{n_{j_1}} n_{i_1}&=[v_1']_{n_{j_1}}+u_1'' n_{j_1}^{|v_1|} \quad (u_1, u_1''\in[n_{i_1}], v_1,v_1'\in [n_{j_1}]^{|v_1|}),\\
u_1''+\ft n_{i_1}&=\ft'+u_1' n_{j_1} \quad (u_1', u_1''\in[n_{i_1}], \ft, \ft''\in [n_{j_1}]).
\end{align*}
Then multiply the second identity above by $n_{j_1}^{|v_1|}$ and then add it to the first one above yields
\[
u_1+([v_1]_{n_{j_1}}+\ft n_{j_1}^{|v_1|}) n_{i_1}=([v_1']_{n_{j_1}}+\ft'n_{j_1}^{|v_1|})+u_1' n_{j_1}^{|v_1\ft|}.
\]
This proves the lemma for $u_1\in \Lambda^{e_{i_1}}$ and any $v_1\in \Lambda_{j_1}$. Apply induction again to get that the lemma holds any $u\in\Lambda_{i_1}$
and $v\in\Lambda_{j_1}$.

A similar argument applies to general $p,q\in \bN$.

\smallskip
To see the uniqueness solution of \eqref{E:odometer}, let
\[
M:=\sum_{\ell=1}^p[u_\ell]_{n_{i_\ell}}\mathbf{n}^{|(u_1,\ldots, u_{\ell-1})|}, \ N:=\sum_{\fk=1}^q[v_\fk]_{n_{i_\ell}}\mathbf{n}^{|(v_1,\ldots, v_{\fk-1})|}.
\]
Then simple calculations can show that
\[
M \in [\mathbf{n}^{|(u_1,\ldots, u_p)|}],\ N\in [\mathbf{n}^{|(v_1,\ldots, v_q)|}].
\]
So the left hand side of \eqref{E:odometer} is a number in $[\mathbf{n}^{|(v_1,\ldots, v_q)|}\mathbf{n}^{|(v_1,\ldots, v_q)|}]$. Hence
there are unique numbers $M'\in  [\mathbf{n}^{|(u_1,\ldots, u_p)|}]$ and $N'\in  [\mathbf{n}^{|(v_1,\ldots, v_q)|}]$ such that
\[
M+N \mathbf{n}^{|(u_1,\ldots, u_p)|}=N'+M' \mathbf{n}^{|(v_1,\ldots, v_q)|}.
\]
Then apply divisions to $M'$ consecutively by $\mathbf{n}^{|(u_1,\ldots, u_{p-1})|},\ldots,  \mathbf{n}^{|(u_1)|}$  to obtain unique $u_1', \ldots, u_p'$.
Likewise for $N'$.
\end{proof}

The above lemma exactly says that all ``higher-level" commutation relations are still odometers. So the following corollary is now immediate.

\begin{cor}
\label{C:odometer}
Let $p,q\in\bN^k$ and $n^p=n^q$. If $u_\ell\in\Lambda^{p_\ell e_{i_\ell}}$ and $v_\ell\in\Lambda^{q_\ell e_{j_\fk}}$ for $1\le \ell\le k$. Then
 there is a bijection
$
\gamma:\Lambda^p\to \Lambda^q$, $x_{u_1}^{i_1}\cdots x_{u_p}^{i_p}\mapsto x_{v_1'}^{j_1}\cdots x_{v_q'}^{j_q}
$
such that
\[
x_{u_1}^{i_1}\cdots x_{u_p}^{i_p}\, x_{v_1}^{j_1}\cdots x_{v_q}^{j_q}
=\gamma(x_{u_1}^{i_1}\cdots x_{u_p}^{i_p})\, \gamma^{-1}(x_{v_1}^{j_1}\cdots x_{v_q}^{j_q}).
\]
\end{cor}

\begin{proof}
Notice that $\mathbf{n}^{|(u_1,\ldots, u_{\ell-1})|}=n^p$ and $\mathbf{n}^{|(v_1,\ldots, v_q)|}=n^q$.
So by Lemma \ref{L:odometer},
\begin{align*}
&\sum_{\ell=1}^p[u_\ell]_{n_{i_\ell}}\mathbf{n}^{|(u_1,\ldots, u_{\ell-1})|}
+\left(\sum_{\fk=1}^q[v_\fk]_{n_{j_\fk}}\mathbf{n}^{|(v_1,\ldots, v_{\fk-1})|}\right)n^p\\
=&\sum_{\fk=1}^q[v_\fk']_{n_{j_\fk}}\mathbf{n}^{|(v_1,\ldots, v_{\fk-1})|}
  +\left(\sum_{\ell=1}^p[u_\ell']_{n_{i_\ell}}\mathbf{n}^{|(u_1,\ldots, u_{\ell-1})|}  \right)n^p.
\end{align*}
So
\[
\sum_{\fk=1}^q[v_\fk']_{n_{j_\fk}}\mathbf{n}^{|(v_1,\ldots, v_{\fk-1})|}=\sum_{\ell=1}^p[u_\ell]_{n_{i_\ell}}\mathbf{n}^{|(u_1,\ldots, u_{\ell-1})|}
\]
and
\[
\sum_{\ell=1}^p[u_\ell']_{n_{i_\ell}}\mathbf{n}^{|(u_1,\ldots, u_{\ell-1})|}=\sum_{\fk=1}^q[v_\fk]_{n_{j_\fk}}\mathbf{n}^{|(v_1,\ldots, v_{\fk-1})|}.
\]
As mentioned before, $v_1', \ldots, v_q'$ are uniquely determined by $u_1, \ldots, u_p$.
Thus there is a bijection
$
\gamma:\Lambda^p\to \Lambda^q$, $x_{u_1}^{i_1}\cdots x_{u_p}^{i_p}\mapsto x_{v_1'}^{j_1}\cdots x_{v_q'}^{j_q}
$
such that
\[
x_{u_1}^{i_1}\cdots x_{u_p}^{i_p}\, x_{v_1}^{j_1}\cdots x_{v_q}^{j_q}
=\gamma(x_{u_1}^{i_1}\cdots x_{u_p}^{i_p})\, \gamma^{-1}(x_{v_1}^{j_1}\cdots x_{v_q}^{j_q}).
\]
We are done.
\end{proof}

\begin{thm}
\label{T:odoPer1}
$\Per_\Lambda=\{p\in\bZ^k: n^p=1\}$.
\end{thm}

\begin{proof}
This follows from Corollary \ref{C:odometer} and \cite[Theorem 7.1]{DY09}.
\end{proof}

\end{document}